\documentclass{scrartcl}
\usepackage[a4paper,top=25mm,bottom=25mm, right=20mm, left=20mm]{geometry}
\usepackage[utf8]{inputenc}
\usepackage[english]{babel}
\usepackage{amsmath, mathtools}
\usepackage{amsfonts}\usepackage{amssymb}
\usepackage{amsthm}
\usepackage{graphicx}

\usepackage{physics}
\usepackage{tikz}
\usepackage{mathdots}
\usepackage{yhmath}
\usepackage{cancel}
\usepackage{color}
\usepackage{siunitx}
\usepackage{array}
\usepackage{multirow}
\usepackage{gensymb}
\usepackage{tabularx}
\usepackage{extarrows}
\usepackage{booktabs}
\usepackage{subcaption}
\usepackage{wrapfig}
\usepackage{algorithm}
\usepackage{algpseudocode}
\usepackage{listings}
\usepackage{xcolor}
\usepackage{url}
\usetikzlibrary{fadings}
\usetikzlibrary{patterns}
\usetikzlibrary{shadows.blur}
\usetikzlibrary{shapes}
\usetikzlibrary{calc}
\usetikzlibrary{positioning}
\usepackage[breaklinks=true]{hyperref}

\colorlet{RefColor}{green!50!black}
\colorlet{LinkColor}{red!50!black}
\hypersetup{
  colorlinks = true,
  linkbordercolor = {white},
  linkcolor = LinkColor,
  anchorcolor=LinkColor,
  citecolor= RefColor,
  filecolor=cyan,
  menucolor=red,
  runcolor=cyan,
  urlcolor = LinkColor,
}

\usepackage{appendix}
\usepackage{amssymb}
\usepackage{booktabs}
\usepackage{todonotes}
\usepackage{etoolbox}

\definecolor{myred}{RGB}{230,75,53}
\definecolor{myblue}{RGB}{77,187,213}
\definecolor{mygreen}{RGB}{0,160,135}

\definecolor{mydarkblue}{RGB}{74, 174, 199}
\usepackage{macros}

\numberwithin{equation}{section}
\title{Nonlinear model reduction with Neural Galerkin schemes on quadratic manifolds\thanks{The authors were partially supported by the National Science Foundation under Grant No.~2046521. Furthermore, the work is funded by the U.S.~Department of Energy, Office of Science Energy Earthshot Initiative as part of the project "Learning reduced models under extreme data conditions for design and rapid decision-making in complex systems" under Award \#DE-SC0024721.}}
\author{{Philipp Weder}\thanks{Chair for Numerical Modelling and Simulation, École Polytechnique Fédérale de Lausanne}\and{Paul Schwerdtner}\thanks{Courant Institute of Mathematical Sciences, New York University (corresponding author \url{paul.schwerdtner@nyu.edu}).}
\and {Benjamin Peherstorfer}\footnotemark[3]}
\date{December 2024}

\begin{document}

\maketitle

\begin{abstract}
Leveraging nonlinear parametrizations for model reduction can overcome the Kolmogorov barrier that affects transport-dominated problems. 
In this work, we build on the reduced dynamics given by Neural Galerkin schemes and propose to parametrize the corresponding reduced solutions on quadratic manifolds. 
We show that the solutions of the proposed quadratic-manifold Neural Galerkin reduced models are locally unique and minimize the residual norm over time, which promotes stability and accuracy. 
For linear problems, quadratic-manifold Neural Galerkin reduced models achieve online efficiency in the sense that the costs of predictions scale independently of the state dimension of the underlying full model. For nonlinear problems, we show that Neural Galerkin schemes allow using separate collocation points for evaluating the residual function from the full-model grid points, which can be seen as a form of hyper-reduction. 
Numerical experiments with advecting waves and densities of charged particles in an electric field show that quadratic-manifold Neural Galerkin reduced models lead to orders of magnitude speedups compared to full models. 
\end{abstract}

\textbf{Keywords:}
   nonlinear model reduction, Neural Galerkin schemes, quadratic manifolds, Kolmogorov barrier

\section{Introduction}
Leveraging nonlinear parametrizations for model reduction \cite{RozzaPateraSurvey,SIREVSurvey,interpbook} enables reducing transport-dominated and other problems that are affected by the Kolmogorov barrier \cite{Peherstorfer2022Breaking}. 
In this work, we build on Neural Galerkin schemes \cite{BRUNA2024112588,BERMAN2024389,zhang2024} that use nonlinear parametrizations to approximate solution fields. 
We show that Neural Galerkin schemes can be used together with nonlinear parametrizations corresponding to quadratic-manifold approximations \cite{JainTRR2017quadratic,GeelenWW2023Operator,BarnettF2022Quadratic}. Quadratic approximations have rich structure that can be leveraged by Neural Galerkin schemes: 
We show that the proposed quadratic-manifold Neural Galerkin (QMNG) reduced models have locally unique solutions under standard assumptions on the underlying full models, and so avoid the tangent-space collapse phenomenon that other nonlinear parametrization can suffer from \cite{zhang2024}. Furthermore, the application of Neural Galerkin schemes guarantees that QMNG reduced solutions minimize the residual norm, which leads to stabler behavior than other ways of defining reduced dynamics on quadratic manifolds as our experiments show. 
Applying Neural Galerkin schemes to quadratic manifolds also separates collocation points at which the residual norm is evaluated from the grid points of the underlying full model. This allows for choosing collocation points that are different from the grid points of the full model, which is a form of hyper-reduction for nonlinear problems. Thus, adding another level of approximation with  empirical interpolation \cite{BarraultMNP2004Empirical} and related methods can be avoided in QMNG reduced models. 
Experiments with a range of numerical examples demonstrate that QMNG reduced models can achieve orders of magnitude speedups over full models of advecting waves and other transport-dominated phenomena. In particular, if the full model is linear, then QMNG reduced models leverage the specific structure of quadratic manifolds to achieve online efficiency in the sense that the costs of predictions with QMNG reduced models scale independently of the costs of the full model once the QMNG reduced model has been constructed.

We refer to \cite{Peherstorfer2022Breaking} for a brief survey on nonlinear model reduction and review only the works that are closest to our approach of using Neural Galerkin schemes with quadratic manifolds. 
Neural Galerkin schemes have been developed for numerically solving partial differential equations with neural-network parametrizations \cite{BRUNA2024112588}. Neural Galerkin schemes build on the Dirac-Frenkel variational principle \cite{dirac1930note,frenkel1934wave,lubich2008quantum}. The dynamics given by the Dirac-Frenkel variational principle are also used for dynamic low-rank approximations \cite{doi:10.1137/050639703} and in the works \cite{UngerTransformModes2020,Du_2021,doi:10.1137/21M1415972,kast2023positional,schulze2023structurepreservingmodelreductionporthamiltonian}; we also refer to the survey \cite{BERMAN2024389} and the works \cite{lasser_lubich_2020,HesthavenPR2022Reduced} for additional references. Similar concepts of finding reduced dynamics are also explored in \cite{LeeC2020Model,KIM2022110841,doi:10.2514/1.J059785,cocola2023hyperreducedautoencodersefficientaccurate,Romor2023} where neural network-based autoencoders are used for obtaining nonlinear parametrizations, which typically have to be equipped with hyper-reduction for achieving runtime speedups. 
Neural Galerkin schemes have been used with pre-trained nonlinear parametrization such as neural networks with continuous low-rank adaptation layers \cite{berman2024colora}. There is also the work \cite{NEURIPS2023_0cb310ed} that shows that only sparse subsets of the set of all weights of neural networks need to be updated during time integration in Neural Galerkin schemes, which can be seen as a form of training-free model reduction. 

We propose to use quadratic approximations with Neural Galerkin schemes in the following: Quadratic manifolds have been first used in model reduction in \cite{JainTRR2017quadratic,RutzmoserRTJ2017Generalization} and since then have led to a series of works that address intrusive and non-intrusive reduced modeling with quadratic manifolds as well as constructing quadratic manifolds from training data \cite{GeelenWW2023Operator,BarnettF2022Quadratic,GeelenBW2023Learning,GeelenBWW2023Learning,SHARMA2023116402,CRMECA_2023__351_S1_357_0,BennerGHP-D2023quadratic,SchwerdtnerP2024Greedy,schwerdtner2024onlinelearningquadraticmanifolds,GoyalB2024Generalized,Ballout2024,schwerdtner24qmsr}. 
The work \cite{BarnettF2022Quadratic} learns reduced models on quadratic manifolds from data in a non-intrusive way, i.e., without requiring intrusive access to the underlying full-model solver that generates the training data. However, for easing the fitting of the reduced model to data, the reduced dynamics are defined by setting the residual orthogonal to a constant-in-time test space. While this choice is convenient for the non-intrusive setting that is considered in \cite{GeelenWW2023Operator}, we will show that in the intrusive setting, when the full model is available, setting the residual orthogonal to the column space of the Jacobian of the parametrization as in Neural Galerkin schemes can lead to more stable predictions. The stabler behavior can be especially well seen on quadratic manifolds that are tightly fit to the training data and can be explained by observing that setting the residual orthogonal to the column space of the Jacobian leads to the minimization of the residual norm. We stress, however, that using a constant-in-time test space can lead to lower runtimes, as we show in our numerical experiments. Furthermore, it remains an open question how to use Neural Galerkin schemes in non-intrusive settings. 
The work \cite{SHARMA2023116402} derives intrusive reduced models on quadratic manifolds but it focuses on Hamiltonian systems and the conservation of Hamiltonians over time. Nevertheless, it is also proposed to use a test space that varies with time against which the residual is set to orthogonal.  
However, the trial and test spaces used in \cite{SHARMA2023116402} are motivated by preserving Hamiltonians and thus have specific structure corresponding to Hamiltonian systems. In particular, position and momentum have to be coupled, which leads to online calculations with cost complexities that increase quicker with the reduced dimension than in our approach. 
The work \cite{JainTRR2017quadratic} also uses a time-varying test space, even though the focus of that work is on second-order differential equations and the quadratic manifolds are trained specifically for these systems. Furthermore, we will show for quadratic-manifold Neural Galerkin reduced models that the solutions minimize the residual norm and are locally unique.   
The works \cite{BarnettF2022Quadratic,BarnettFM2023Neural-network-augmented} consider quadratic manifolds and more generic manifolds with data-driven feature maps. The reduced dynamics are obtained in the fully discrete setting by minimizing the time-discrete residual over time. This is related to Neural Galerkin schemes in the sense that Neural Galerkin schemes minimize the time-continuous residual. We refer to \cite{zhang2024} for an in-depth discussion about minimizing the time-continuous versus the time-discrete residual when training nonlinear parametrizations in the context of time-dependent partial differential equations.

The manuscript is organized as follows. We discuss Neural Galerkin schemes and quadratic manifolds in the preliminaries in Section~\ref{sec:Prelim}. We then introduce QMNG reduced models based on quadratic parametrizations and interpolation in Section~\ref{sec:QMNG}. In Section~\ref{sec:QMNGLinear} we derive QMNG reduced models in the special case that the collocation points of Neural Galerkin schemes and the grid points of the full model coincide, which allows pre-computing quantities and achieving online efficiency in the sense that prediction costs for linear problems scale independently of the number of degrees of freedom of the full model. Numerical experiments are presented in Section~\ref{sec:NumExp} and Conclusions are drawn in Section~\ref{sec:Conc}.

\section{Preliminaries}\label{sec:Prelim}
We briefly describe the setup that we consider and then discuss Neural Galerkin schemes and quadratic manifolds.

\subsection{Setup}\label{sec:Prelim:Setup}
Consider a time-dependent partial differential equation (PDE)
\begin{equation}\label{eq:Prelim:PDE}
\begin{aligned}
\partial_t q(\bfx, t; \bfmu) &= f(\bfx, q; \bfmu)\,,\\
q(\bfx, 0; \bfmu) & = q_0(\bfx; \bfmu)\,,
\end{aligned}
\end{equation}
with solution field $q: \Omega \times [0, T] \times \Dcal \to \mathbb{R}$ that evolves over time $t \in [0, T]$, spatial coordinate $\bfx \in \Omega \subset \mathbb{R}^d$, and physics parameter $\bfmu \in \Dcal \subset \mathbb{R}^{d^{\prime}}$. The dynamics are given by the right-hand side function $f$. Evaluating the function $f$ can require taking partial derivatives of $q$ with respect to the spatial coordinate $\bfx$. The initial condition is $q_0: \Omega \times \Dcal \to \mathbb{R}$. Note that the solution field $q$ is scalar-valued for ease of exposition. All of the following can be extended to vector-valued solution fields.

\subsection{Neural Galerkin schemes}\label{sec:Prelim:NG}
For applying Neural Galerkin schemes to numerically solve \eqref{eq:Prelim:PDE}, we consider a nonlinear parametrization $\tilde{q}: \Theta \times \Omega \to \mathbb{R}$ that depends on a weight vector $\bftheta(t; \bfmu) \in \Theta \subseteq \mathbb{R}^\np$, which depends on time $t$ and the physics parameter $\bfmu$. The parametrization is nonlinear in the sense that the weights $\bftheta(t; \bfmu)$ can enter nonlinearly into the function $\tilde{q}$, which is in stark contrast to a wide range of numerical methods in numerical analysis that rely on the weights to enter linearly; see, e.g., \cite{zhang2024} for a discussion. The function $\tilde{q}$ is a time-dependent parametrization of $q$ because the weight $\bftheta(t; \bfmu)$ is time-dependent.
The nonlinear parametrization $\tilde{q}$ can be generic, such as a neural network with time-dependent and parameter-dependent weights and biases as in \cite{Du_2021,BRUNA2024112588,kast2023positional,berman2024colora}.
Plugging $\tilde{q}$ into the PDE \eqref{eq:Prelim:PDE} leads to the residual function $r: \Theta \times \Theta \times \Omega \times \Dcal \to \mathbb{R}$, 
\begin{equation}\label{eq:Prelim:ResFun}
r(\bftheta(t; \bfmu), \dot{\bftheta}(t; \bfmu), \bfx; \bfmu) = \nabla_{\bftheta} \tilde{q}(\bftheta(t; \bfmu), \bfx) \cdot \dot{\bftheta}(t; \bfmu) - f(\bfx, \tilde{q}(\bftheta(t; \bfmu), \cdot); \bfmu)\,,
\end{equation}
where $\dot{\bftheta}(t; \bfmu)$ denotes the time derivative of $\bftheta(t; \bfmu)$ that is obtained by applying the chain rule to $\partial_t\tilde{q}(\bftheta(t; \bfmu), \cdot) = \nabla_{\bftheta}\tilde{q}(\bftheta(t; \bfmu), \cdot) \cdot \dot{\bftheta}(t; \bfmu)$. 

In Neural Galerkin schemes, the time derivative $\dot{\bftheta}(t; \bfmu)$ is determined via the Dirac-Frenkel variational principle \cite{dirac1930note,frenkel1934wave,lubich2008quantum,doi:10.1137/050639703} by setting the residual \eqref{eq:Prelim:ResFun} orthogonal to the test space spanned by the components of the gradient $\nabla_{\bftheta}\tilde{q}(\bftheta(t; \bfmu), \cdot)$,
\begin{equation}\label{eq:Prelim:ResEquation}
\langle \partial_{\theta_i}\tilde{q}(\bftheta(t; \bfmu), \cdot), r(\bftheta(t; \bfmu), \dot{\bftheta}(t; \bfmu), \cdot; \bfmu)\rangle = 0\,,\qquad i = 1, \dots, n\,,
\end{equation}
where the inner product $\langle \cdot, \cdot \rangle$ typically is the $L^2$ inner product over $\Omega$; see \cite{BRUNA2024112588,BERMAN2024389} for details.
Transforming the system of equations \eqref{eq:Prelim:ResEquation} reveals that they are linear in the unknown $\dot{\bftheta}(t; \bfmu)$ and that they are the normal equations corresponding to the least-squares problem
\begin{equation}\label{eq:Prelim:NGLSQ}
\min_{\dot{\bftheta}(t; \bfmu) \in \Theta} \|\nabla_{\bftheta}\tilde{q}(\bftheta(t; \bfmu), \cdot) \cdot \dot{\bftheta}(t; \bfmu) - f(\cdot, \tilde{q}(\bftheta(t; \bfmu), \cdot); \bfmu)\|^2\,,
\end{equation}
where the norm $\|\cdot\|$ is induced by the inner product $\langle \cdot, \cdot \rangle$. 
Solving the least-squares problem \eqref{eq:Prelim:NGLSQ} over time $t \in [0, T]$ for a given parameter $\bfmu \in \Dcal$ provides a weight trajectory $t \mapsto \bftheta(t; \bfmu)$ so that the corresponding function $\tilde{q}(\bftheta(t; \bfmu), \cdot)$ numerically solves the PDE \eqref{eq:Prelim:PDE} in the sense of having an orthogonal residual given by the conditions given in  \eqref{eq:Prelim:ResEquation}.  
We refer to \cite{lubich2008quantum,doi:10.1137/050639703} for discussions about the Dirac-Frenkel variational principle and to \cite{BERMAN2024389,zhang2024} for Neural Galerkin schemes and applications to neural-network parametrizations.

\subsection{Nonlinear parametrizations corresponding to quadratic manifolds}\label{sec:Prelim:QM}
Dimensionality reduction is typically described via an encoder function $e: \mathbb{R}^{N} \to \mathbb{R}^{n}$ that maps a data point $\bfs \in \mathbb{R}^{N}$ in the high-dimensional space $\mathbb{R}^N$ onto a low-dimensional data point $\bhats \in \mathbb{R}^n$ in $\mathbb{R}^n$ with lower dimension $n \ll N$, and a decoder function $g: \mathbb{R}^n \to \mathbb{R}^N$ that lifts a low-dimensional data point $\bhats$ in $\mathbb{R}^n$ back onto a point in the high-dimensional space $\mathbb{R}^N$.
Dimensionality reduction on quadratic manifolds \cite{RutzmoserRTJ2017Generalization,JainTRR2017quadratic,BarnettF2022Quadratic,GeelenBW2023Learning,GeelenWW2023Operator,SchwerdtnerP2024Greedy} uses an affine encoder function 
\[
e_{\bfV}(\bfs) = \bfV^{\top}(\bfs - \bfs_0),
\]
where $\bfs_0 \in \mathbb{R}^N$ is a reference point and $\bfV \in \mathbb{R}^{N \times n}$ is a basis matrix with orthonormal columns, and a quadratic decoder function 
\begin{equation}\label{eq:Prelim:DecoderQuad}
g_{\bfV,\bfW}(\bhats) = \boldsymbol{s}_0 + \bfV\bhats + \bfW h(\bhats)\,.
\end{equation}
The decoder function is quadratic because the feature map $h: \mathbb{R}^{n} \to \mathbb{R}^{n^2}$ is a quadratic function,
\[
h: \mathbb{R}^{n} \to \mathbb{R}^{n^2}, \quad \bhats \mapsto \bhats \otimes \bhats \coloneqq \begin{bmatrix} \hat{s}_1\hat{s}_1 & \hat{s}_1\hat{s}_2 & \dots & \hat{s}_1\hat{s}_n & \hat{s}_2\hat{s}_2 \dots \hat{s}_n\hat{s}_n\end{bmatrix}^{\top}\,,
\]
where $\bhats = [\hat{s}_1, \dots, \hat{s}_n]^{\top} \in \mathbb{R}^n$. We point out that one could remove the duplicate entries from $h(\bhats)$ to obtain a quadratic feature map with reduced output dimension $n(n+1)/2$ as it is done in \cite{GeelenBWW2023Learning}. For ease of exposition, we adhere to $h$, but the following arguments also apply to its reduced variant.
The decoder function $g_{\bfV,\bfW}$ depends on the same reference point $\bfs_0$ and basis matrix $\bfV$ as the encoder function $e_{\bfV}$ and additionally on the matrix $\bfW \in \mathbb{R}^{N \times n^2}$.  
While we focus on decoder functions with quadratic feature maps, all of the following methodology can be generalized to other, more general, feature maps \cite{BarnettFM2023Neural-network-augmented,CRMECA_2023__351_S1_357_0}. Applying the decoder function defines a subset in $\mathbb{R}^N$,
\[
\mathcal{M}_{n} = \{g_{\bfV, \bfW}(\bhats) \,|\, \bhats \in \mathbb{R}^n\}\subset \mathbb{R}^N\,,
\]
to which we refer as quadratic manifold in $\mathbb{R}^N$ because the decoder is quadratic in its argument. We stress that $\mathcal{M}_n$ does not necessarily have a manifold structure but we still use the term manifold to be in agreement with the terminology in the literature \cite{BarnettF2022Quadratic,GeelenBW2023Learning,GeelenWW2023Operator}. 

Given training data points $\bfs_1, \dots, \bfs_M \in \mathbb{R}^N$, the matrices $\bfV$ and $\bfW$ that define a quadratic manifold can be trained, e.g., via the greedy approach introduced in \cite{SchwerdtnerP2024Greedy}.
The greedy method introduced in \cite{SchwerdtnerP2024Greedy} provides an orthonormal basis matrix $\bfV$ by taking into account leading and later left-singular vectors of the data matrix $\bfS = [\bfs_1, \dots, \bfs_M] \in \mathbb{R}^{N \times M}$. Once the matrix $\bfV$ has been obtained by the greedy method, the matrix $\bfW$ is fitted via least-squares regression as
\begin{equation}\label{eq:Prelim:FitW}
\bfW = \operatorname*{arg\,min}_{\bfW \in \mathbb{R}^{N \times n^2}} \|\bfV\bfV^{\top}\bfS - \bfS + \bfW h(e_{\bfV}(\bfS))\|_F^2 + \gamma \|\bfW\|_F^2\,,
\end{equation}
where we overload the notation of the encoder function $e_{\bfV}$ to operate column-wise on the matrix $\bfS$. 
The factor $\gamma > 0$ is a regularization parameter that can help prevent the overfitting of $\bfW$ to the training data. In~\cite[Section 2.4]{GeelenWW2023Operator} it is noted that fitting $\bfW$ with \eqref{eq:Prelim:FitW} guarantees $\bfW^\top \bfV=0$, which we will use later.

\section{Nonlinear model reduction with Neural Galerkin schemes on quadratic manifolds}\label{sec:QMNG}
We propose to leverage approximations on quadratic manifolds together with Neural Galerkin schemes to derive reduced models. 
Combining quadratic manifolds and Neural Galerkin schemes leads to reduced models that minimize the residual at each point in time to obtain a Galerkin-optimal solution on the quadratic manifold. At the same time, the quadratic approximations have a rich structure that is inherited by the systems that need to be solved in each time step of Neural Galerkin schemes. We also relate QMNG reduced models to other reduced models in the literature \cite{GeelenWW2023Operator,SHARMA2023116402} and show connections with respect to the test space against which the residual is set orthogonal. %

\subsection{Training quadratic manifolds on snapshot data}\label{sec:QMNG:Snapshots}
We train the quadratic manifold on snapshot data obtained from numerical approximations of the solution fields $q$ of \eqref{eq:Prelim:PDE}, which are derived from a semi-discrete problem that we refer to as the full model, 
\begin{equation}\label{eq:Prelim:SemiDiscPDE}
\begin{aligned}
\partial_t \bfq(t; \bfmu) &= \bff(\bfq(t; \bfmu); \bfmu)\,,\\
\bfq(0; \bfmu) & = \bfq_0(\bfmu)\,.
\end{aligned}
\end{equation}
The state $\bfq: [0, T] \times \Dcal \to \mathbb{R}^N, \bfq(t; \bfmu) = [q_1(t; \bfmu), \dots, q_N(t; \bfmu)]^T$ is a vector-valued function with $N \in \mathbb{N}$ component functions $q_1, \dots, q_N$. The $N$ component functions correspond to approximations of the solution field $q$ of \eqref{eq:Prelim:PDE} evaluated at the corresponding grid points in the set \begin{equation}\label{eq:QMNG:SpatialCoordinates}
\mathcal{X} = \{\bfx_1, \dots, \bfx_N\} \subset \Omega\,.
\end{equation}
The right-hand side function $\bff: \mathbb{R}^{N} \times \Dcal \to \mathbb{R}^{N}$ and initial condition $\bfq_0: \Dcal \to \mathbb{R}^N$ are analogously defined.

Let now $\bfmu_1, \dots, \bfmu_{M'} \in \mathcal{D}$ be training parameters and let $t_1, \dots, t_K \in [0, T]$ be time points at which numerical solutions to the semi-discrete PDE problem \eqref{eq:Prelim:SemiDiscPDE} are available as offline data in a snapshot matrix,
\begin{equation}\label{eq:QMNG:SnapshotQ}
\bfQ = \begin{bmatrix}
| & & | & | & & |\\
\bfq(t_1; \bfmu_1) & \dots & \bfq(t_K; \bfmu_1) & \bfq(t_1; \bfmu_2) & \dots & \bfq(t_K; \bfmu_{M'})\\
| & & | & | & & |
\end{bmatrix} \in \mathbb{R}^{N \times KM'}\,.
\end{equation}
The encoder function $e_{\bfV}$ and decoder function $g_{\bfV,\bfW}$ of a quadratic manifold as described in Section~\ref{sec:Prelim:QM} are then trained on the snapshot data $\bfQ$.
In particular, we obtain $\bfV$ and $\bfW$ with the greedy method introduced in~\cite{SchwerdtnerP2024Greedy}.

\subsection{Interpolated quadratic parametrizations}
Let $e_{\bfV}: \mathbb{R}^N \to \mathbb{R}^n$ and $g_{\bfV,\bfW}: \mathbb{R}^n \to \mathbb{R}^N$ be encoder and decoder functions corresponding to a quadratic manifold of dimension $n$ that are trained on snapshot data $\bfQ$.
We introduce an interpolation operator $I: \Omega \times \mathbb{R}^{N \times d} \times \mathbb{R}^{N \times l} \to \mathbb{R}^l$, which takes a coordinate $\bfx \in \Omega$ in the spatial domain and the $N$ grid points $\bfx_1, \dots, \bfx_N$ given in \eqref{eq:QMNG:SpatialCoordinates} in matrix form $\bfX = [\bfx_1, \dots, \bfx_N]^{\top} \in \mathbb{R}^{N \times d}$. 
The third argument is a matrix 
\[
\bfU = \begin{bmatrix} u_1(\bfx_1) & \dots & u_l(\bfx_1)\\ 
\vdots & & \vdots\\
u_1(\bfx_N) & \dots & u_l(\bfx_N)\end{bmatrix} \in \mathbb{R}^{N \times l}
\]
of $N$ function evaluations at the $N$ grids points of $l$ functions $u_1, \dots, u_l: \Omega \to \mathbb{R}$. The output of $I$ is a vector of approximations of the values of the $l$ functions at the coordinate $\bfx$ in the spatial domain $\Omega$,
\[
I(\bfx, \bfX, \bfU) \approx \begin{bmatrix} 
u_1(\bfx)\\
\vdots\\
u_l(\bfx)
\end{bmatrix} \in \mathbb{R}^l\,.
\]
Given the matrices $\bfV$ and $\bfW$ corresponding to the encoder $e_{\bfV}$ and the decoder $g_{\bfV,\bfW}$ functions, we can interpret the columns of $\bfV$ and $\bfW$ as corresponding to function values, which we interpolate with $I$. 
Correspondingly, the decoder $g_{\bfV,\bfW}$ induces a quadratic parametrization
\begin{equation}\label{eq:QMNG:ContEncoder}
g_{\bfV,\bfW,I}: \Theta \times \Omega \to \mathbb{R}\,,\qquad  g_{\bfV,\bfW,I}(\bftheta, \bfx)=I(\bfx, \bfX, \bfs_0) + I(\bfx, \bfX, \bfV)^{\top}\bftheta + I(\bfx, \bfX, \bfW)^{\top}h(\bftheta)\,,
\end{equation}
with the $\bfs_0, \bfV,$ and $\bfW$ of the decoder function $g_{\bfV, \bfW}$ defined in \eqref{eq:Prelim:DecoderQuad} in Section~\ref{sec:Prelim:QM} and obtained with the training procedure discussed in Section~\ref{sec:QMNG:Snapshots}. The notation of $I$ is overloaded in \eqref{eq:QMNG:ContEncoder} because the matrices $\bfs_0, \bfV, \bfW$ have different numbers of columns. 
Importantly, the parametrization $g_{\bfV,\bfW,I}$ can be evaluated at any point in the spatial domain $\Omega$, whereas the decoder function $g_{\bfV, \bfW}$ given in \eqref{eq:Prelim:DecoderQuad} provides a decoded state corresponding to the grid points $\bfx_1, \dots, \bfx_N$ of the full model only. 

\subsection{Applying Neural Galerkin schemes to interpolated quadratic parametrizations}
The parametrization $g_{\bfV,\bfW,I}$ has the same signature as nonlinear parametrization $\tilde{q}$ used in Section~\ref{sec:Prelim:NG} and thus Neural Galerkin schemes are directly applicable; note that we will discuss an extension of Neural Galerkin schemes that are applicable directly to $g_{\bfV,\bfW}$ in Section~\ref{sec:QMNGLinear}.  
By applying Neural Galerkin to $g_{\bfV,\bfW,I}$, we seek $\dot{\bftheta}(t)$ such that the residual \eqref{eq:Prelim:ResFun} with $g_{\bfV,\bfW,I}$ is orthogonal in terms of \eqref{eq:Prelim:ResEquation}, which leads to the least-squares problem \eqref{eq:Prelim:NGLSQ}. 
The objective function of the least-squares problem \eqref{eq:Prelim:NGLSQ} is formulated via the $L^2$ norm, which needs to be numerically approximated. 
Following standard convention \cite{BRUNA2024112588}, we approximate the $L^2$ norm in the objective by choosing a set $\Xi = \{\bfxi_1, \dots, \bfxi_m\} \subset \Omega$ of $m$ collocation points $\bfxi_1, \dots, \bfxi_m$ and solving
\begin{equation}\label{eq:QMNG:NGXi}
\min_{\dot{\bftheta}(t) \in \Theta} \| \bfJ_{\Xi}(\bftheta(t))\dot{\bftheta}(t) - \bff_{\Xi}(\bftheta(t); \bfmu)\|_2^2
\end{equation}
with the batch Jacobian matrix
\begin{equation}\label{eq:QMNG:JacobianXi}
\bfJ_{\Xi}(\bftheta(t; \bfmu)) = \begin{bmatrix}
\frac{\partial}{\partial \theta_1} g_{\bfV,\bfW,I}(\bfxi_1, \bftheta(t; \bfmu)) & \dots & \frac{\partial}{\partial \theta_n}g_{\bfV, \bfW, I}(\bfxi_1, \bftheta(t; \bfmu))\\
\vdots & & \vdots\\
\frac{\partial}{\partial \theta_1}g_{\bfV,\bfW,I}(\bfxi_m, \bftheta(t; \bfmu)) & \dots & \frac{\partial}{\partial \theta_n}g_{\bfV, \bfW, I}(\bfxi_m, \bftheta(t; \bfmu))
\end{bmatrix} \in \mathbb{R}^{m \times n}
\end{equation}
and the batch right-hand side
\begin{equation}\label{eq:QMNG:BatchRHS}
\bff_{\Xi}(\bftheta(t; \bfmu); \bfmu) = \begin{bmatrix} f(\bfxi_1, g_{\bfV, \bfW, I}(\bftheta(t; \bfmu), \cdot); \bfmu) & \cdots & f(\bfxi_m, g_{\bfV, \bfW, I}(\bftheta(t; \bfmu), \cdot); \bfmu)\end{bmatrix}^{\top} \in \mathbb{R}^m\,.
\end{equation}
The choice of the set of collocation points has to be such that the objective of \eqref{eq:QMNG:NGXi} in the Euclidean norm $\|\cdot\|_2$ approximates well the objective of \eqref{eq:Prelim:NGLSQ} in the $L^2$ norm. We consider problems with the spatial domains $\Omega$ of low dimension so that equidistant grid points and uniform sampling in $\Omega$ is sufficient. However, in higher dimensions, selecting the collocation points is more delicate; we refer to  \cite{BRUNA2024112588,WEN2024134129} for extensive discussions about this.

We stress that it is an important feature of our approach that the grid points $\mathcal{X}$ corresponding to the full model and the collocation points $\Xi$ in objective function in the least-squares problem \eqref{eq:QMNG:NGXi} can be different. In particular, the number of collocation points $m$ can be chosen smaller than the number of grid points $N$, which means that our approach has a form of hyper-reduction baked in. Thus, no additional reduction and approximation via empirical interpolation \cite{BarraultMNP2004Empirical} and related methods is necessary. 

The Neural Galerkin conditions encoded in the system \eqref{eq:QMNG:NGXi} provide equations for integrating in time the weight vector $\bftheta(t; \bfmu)$.
We refer to system \eqref{eq:QMNG:NGXi} induced by the Neural Galerkin conditions as the QMNG reduced model, and sometimes as the QMNG reduced model with interpolation to distinguish it from the version that we will introduce in Section~\ref{sec:QMNGLinear} that avoids interpolation.
QMNG reduced models can be discretized in time with a time integration scheme, such as explicit Runge-Kutta schemes.

\subsection{Affine structure of Jacobian matrix of quadratic decoder functions}\label{sec:QMNG:LinAffJ}
The Jacobian matrix $\bfJ_{\Xi}(\bftheta(t; \bfmu))$ defined in \eqref{eq:QMNG:JacobianXi} plays a key role in Neural Galerkin schemes because its columns span the space against which the residual is set orthogonal in the least-squares problem  \eqref{eq:QMNG:NGXi}.
In contrast to generic parametrizations given by neural networks, the Jacobian matrix corresponding to quadratic manifolds has a rich structure that we discuss now. 
Recall that $\bfJ_{\Xi}(\bftheta(t; \bfmu))$ is the Jacobian matrix of the nonlinear parametrization $g_{\bfV,\bfW,I}: \Theta \times \Omega \to \mathbb{R}$ evaluated at the collocation points $\bfxi_1, \dots, \bfxi_m \in \Omega$. 

Because $g_{\bfV,\bfW,I}$ is a quadratic function in the parameter $\bftheta(t; \bfmu)$, its Jacobian matrix is affine
\begin{equation}\label{eq:QMNG:AffineJ}
\bfJ_{\Xi}(\bftheta) = \bfV_{\Xi} + \bfK_{\Xi}(\bftheta)\,,
\end{equation}
where the matrix $\bfV_{\Xi}$ is obtained from $\bfV$ via interpolation as 
\begin{equation}\label{eq:QMNG:VInterp}
\bfV_{\Xi} = \begin{bmatrix}
\text{\,\,---\,\,}  I(\bfxi_1, \bfX, \bfV)^{\top}  \text{\,\,---\,\,}\\
\vdots\\
\text{\,\,---\,\,}  I(\bfxi_m, \bfX, \bfV)^{\top}  \text{\,\,---\,\,}
\end{bmatrix} \in \mathbb{R}^{m \times n}\,,
\end{equation}
The matrix-valued function $\bfK_{\Xi}: \mathbb{R}^n \to \mathbb{R}^{m \times n}$ is 
\begin{equation}\label{eq:QMNG:KMap}
\bfK_{\Xi}(\bftheta) = \bfW_{\Xi} \nabla_{\btheta}h(\btheta) = \bfW_{\Xi}(\bftheta \otimes \bfI + \bfI \otimes \bftheta) \in \mathbb{R}^{m \times n}\,,
\end{equation}
where $\bfW_{\Xi}$ is the interpolated weight matrix obtained from $\bfW$ in an analogous way as \eqref{eq:QMNG:VInterp} is obtained from the basis matrix $\bfV$ and $\bfI$ is the identity matrix of appropriate dimension. 
In tensor format, the function $\bfK_{\Xi}$ is linear in $\bftheta$. To see this, we introduce the following notation for the n-mode product \cite[Section 2.5]{Kolda2009Tensor}. For a three-way tensor $\Tcal  \in \R^{n \times m \times p}$ and a vector $\boldsymbol{v} \in \R^p$, we define the entries of the $n \times m$ matrix $\Tcal \cdot \boldsymbol{v} \in \mathbb{R}^{n \times m}$ as
\begin{equation}\label{eq:QMNG:TensorOne}
    [\Tcal \cdot \boldsymbol{v}]_{ij} = \mathcal{T}_{ijk} v_k\,,\qquad i = 1, \dots, n, \quad j = 1, \dots, m\,.
\end{equation}
The costs of computing \eqref{eq:QMNG:TensorOne} scale as $\mathcal{O}(nmp)$.
Analogously, we define $\bfv \cdot \Tcal \in \mathbb{R}^{n \times m \times p}$ for a vector $\bfv \in \mathbb{R}^{n}$  as   
\begin{equation}\label{eq:QMNG:TensorTwo}
[\boldsymbol{v} \cdot \mathcal{T}]_{jk} = v_i \mathcal{T}_{ijk}\,,\qquad j = 1, \dots, m, \quad k = 1, \dots, p.
\end{equation}
The operation $\cdot$ is linear in both of its arguments.  %

Using the tensor products defined in \eqref{eq:QMNG:TensorOne} and \eqref{eq:QMNG:TensorTwo}, the map $\bfK_{\Xi}$ defined in \eqref{eq:QMNG:KMap} can be written as 
\begin{equation}\label{eq:QMNG:KMapLinTensor}
 \bfK_{\Xi}(\bftheta) = \mathcal{K}_{\Xi} \cdot \bftheta\,,
\end{equation}
with a three-way tensor $\mathcal{K}_{\Xi} \in \mathbb{R}^{m \times n \times n}$. Notice that the three-way tensor $\mathcal{K}_{\Xi}$ depends on the interpolated weight matrix $\bfW_{\Xi}$ obtained from $\bfW$ used in the decoder function $g_{\bfV, \bfW}$ and the collocation points in $\Xi$.
Using the tensor $\mathcal{K}_{\Xi}$ and that the tensor products in \eqref{eq:QMNG:TensorOne} and \eqref{eq:QMNG:TensorTwo} are linear operations, we write $\bfJ_{\Xi}(\bftheta)$ as an affine map in $\bftheta$, 
\begin{equation}\label{eq:QMNG:JTensor}
\bfJ_{\Xi}(\bftheta) = \bfV_{\Xi} + \mathcal{K}_{\Xi} \cdot \bftheta.
\end{equation}

\subsection{Neural Galerkin solutions on quadratic manifolds have minimal residual norm}\label{sec:QMNG:ResNormMin}
The QMNG reduced model sets the residual at $\bftheta(t; \bfmu)$ orthogonal to the space spanned by the columns of the Jacobian \eqref{eq:QMNG:AffineJ} at $\bftheta(t; \bfmu)$, which is equivalent to minimizing the residual in the Euclidean norm over the collocation points $\bfxi_1, \dots, \bfxi_m$ as in the least-squares problem \eqref{eq:QMNG:NGXi}. 
Thus, solutions of QMNG reduced models minimize the residual norm over time $t$ in the sense of \eqref{eq:QMNG:NGXi}. 

Building on the generality of Neural Galerkin schemes, we can use other spaces to set the residual orthogonal to. %
For example, because the Jacobian matrix $\bfJ_{\Xi}(\bftheta(t; \bfmu))$ is affine as shown in Section~\ref{sec:QMNG:LinAffJ}, one could consider a zeroth-order approximation 
\begin{equation}\label{eq:QMNG:ZerothOrderJ}
\bar{\bfJ}_{\Xi}(\bftheta(t; \bfmu)) = \bfV_{\Xi}
\end{equation}
of $\bfJ_{\Xi}(\bftheta(t; \bfmu))$, which is a constant. 
Setting the residual orthogonal to the constant space spanned by the columns of the zeroth-order approximation \eqref{eq:QMNG:ZerothOrderJ} leads to the least-squares problem
\begin{equation}\label{eq:QMNG:OrthoDynLSQ}
\min_{\dot{\bar{\bftheta}}(t; \bfmu) \in \mathbb{R}^n} \|\bfV_{\Xi}\dot{\bar{\bftheta}}(t; \bfmu) - \bff_{\Xi}(\bar{\bftheta}(t; \bfmu); \bfmu)\|_2^2\,,
\end{equation}
with the corresponding solutions
\begin{equation}\label{eq:QMNG:OrthoDyn}
\dot{\bar{\bftheta}}(t; \bfmu) = \bfV^+_{\Xi}\bff_{\Xi}(\bar{\bftheta}(t; \bfmu); \bfmu)\,,
\end{equation}
where $\bfV^+_{\Xi}$ is the Moore-Penrose pseudo-inverse of $\bfV_{\Xi}$. Notice that \eqref{eq:QMNG:OrthoDyn} provides the minimal norm solution of \eqref{eq:QMNG:OrthoDynLSQ} if $\bfV_{\Xi}$ is low rank.
Reduced models based on the dynamics \eqref{eq:QMNG:OrthoDyn} are first introduced in the work \cite{GeelenWW2023Operator} for intrusive and non-intrusive model reduction. Notice that using using $\bfJ_{\Xi}$ directly in non-intrusive model reduction is challenging. Reduced models using dynamics obtained with a time-dependent test space are introduced in \cite{SHARMA2023116402}; however, there the test spaces are more involved as they are constructed such that the corresponding model preserves Hamiltonians. We note that \cite{JainTRR2017quadratic} also use time-varying test spaces but focuses on second-order differential equations. %
The reduced dynamics \eqref{eq:QMNG:OrthoDyn} have the major advantage that the solution to the least-squares problem \eqref{eq:QMNG:OrthoDynLSQ} is computationally cheaper to compute because the zeroth-order approximation $\bar{\bfJ}_{\Xi} = \bfV^{+}_{\Xi}$ of $\bfJ_{\Xi}$ is independent of the weight $\bar{\bftheta}(t; \bfmu)$; however, the corresponding solution does not minimize the norm of the residual, which can lead to instabilities as we will show in the numerical experiments.

\subsection{Algorithms for offline and online phase of QMNG reduced models}%
Algorithm~\ref{alg:offline phase} summarizes the offline steps of QMNG reduced models. The inputs for the offline phase are the snapshot matrix $\bfQ$, the grid points $\mathcal{X}$ of the full model, the reduced dimension $n$, the regularization parameter $\gamma$, and the number of candidate singular vectors $l$ for the greedy construction of the quadratic manifold; for details on the last two parameters, we refer to \cite{SchwerdtnerP2024Greedy}. In a first step, the reference value $\bfs_0$ is computed by taking the mean over all snapshots. Next, the matrices $\bfV$ and $\bfW$ are computed using the greedy quadratic manifold algorithm in \cite{SchwerdtnerP2024Greedy}. Finally, $\bfs_0, \bfV$, and $\bfW$ are interpolated in space and function handles for $I(\cdot, \bfX, \bfV), I(\cdot, \bfX, \bfW)$, and the continuous decoder $g_{\bfV,\bfW,I}$ are returned.

Algorithm~\ref{alg:online phase} summarizes the online phase. We show the algorithm for explicit time integration with the forward Euler method but it generalizes to other time integration schemes. %
The inputs for the online phase are the parameter vector $\bfmu$ at which a solution is to be computed and the corresponding initial condition $\bfq_0(\bfmu)$. Further inputs are the time-steps size $\delta t$ and the number of time steps $K$, as well as the functions $I(\cdot, \bfX, \bfs_0), I(\cdot, \bfX, \bfV), I(\cdot, \bfX, \bfW)$, and $g_{\bfV, \bfW, I}$.
The weight trajectory is initialized by encoding the initial condition $\bfq_0(\bfmu)$ evaluated at the grid points of the full model on the quadratic manifold. 
For each time step, we start by sampling a new batch of collocation points $\bfxi_1, \ldots, \bfxi_m \in \Omega$. This is an optional step and does not have to be done at each time step. Then, we assemble the batch Jacobian matrix  $\bfJ_{\Xi}(\bftheta_k(\bfmu))$ and the batch right-hand side $\bff_{\Xi}(\bftheta_k(\bfmu); \bfmu)$ following \eqref{eq:QMNG:JTensor} and \eqref{eq:QMNG:BatchRHS}, respectively. Finally, we solve the least-squares problem to obtain $\delta{\bftheta}_k(\bfmu)$ and we update the weight vector as $\bftheta_{k + 1}(\bfmu) = \bftheta_k(\bfmu) + \delta t \delta \bftheta_k(\bfmu)$. The algorithm returns the trajectory of the weights $\bftheta_0(\bfmu), \dots, \bftheta_K(\bfmu)$ at the time steps $t_0, \ldots, t_K$.

\subsection{Computational costs}
For the offline phase, the computational costs are dominated by the greedy construction of the quadratic manifold and we refer to \cite{SchwerdtnerP2024Greedy} for a detailed discussion about its cost complexity.

To derive the cost complexity of the online phase, we assume that $I$ corresponds to piecewise cubic spline interpolation, which is also what we use in the numerical experiments. We further assume that the collocation points remain fixed over the time integration. This means that the costs of evaluating the spline interpolant are constant $\mathcal{O}(1)$ if the collocation points coincide with an equidistant grid and scale as $\mathcal{O}(\log(N))$ otherwise. %
To ease exposition, we only consider collocation points that coincide with equidistant grids in $\Omega$ and remain fixed over the time integration so that the cost of evaluating $I$ is constant. Then, the computational cost of computing $\bfV_\Xi$ and $\bfW_\Xi$ at $m$ collocation points scale as $\mathcal{O}(mn)$ and $\mathcal{O}(mn^2)$, respectively. Consequently, the cost of one evaluation of the continuous decoder $g_{\bfV, \bfW, I}$ scales as $\mathcal{O}(mn^2)$.
Taking into account the cost of the tensor product operation $\cdot$, we conclude that the assembly of the batch Jacobian $\bfJ_\Xi$  in each iteration of the time integration loop scale as $\mathcal{O}(mn^2)$.
Similar arguments show that computing the batch right-hand side $\bff_{\Xi}$ scales as $\mathcal{O}(c_fmn^2)$, where $c_f$ is a constant that corresponds to evaluating the right-hand side function $f$ at the function $g_{\bfV, \bfW, I}$, which can be substantial. For example, if automatic differentiation is used, then the costs of the automatic differentiation are captured by $c_f$. If finite-difference approximations of the derivatives in $f$ are used, then the $c_f$ depends on the number of grid points $N$.  
The costs of computing an update $\delta\bftheta_k(\bfmu)$ via the least-squares problem  scale as $\mathcal{O}(c_f m n^2)$ for a direct solver based on the SVD.
Therefore, we conclude that the total costs for one time step of a QMNG reduced model that is discretized with an explicit time integration scheme scale as $\mathcal{O}(c_f m n^2)$.

\begin{algorithm}[t]
\caption{QMNG offline phase}\label{alg:offline phase}
\begin{algorithmic}[1]
\Procedure{QMNGOffline}{$\bfQ, \bfX, n, \gamma, l$}

   \State Apply the greedy method to $\bfQ$ with parameters $n, \gamma, \ell$ to obtain $\bfs_0, \bfV$ and $\bfW$.

   \State Construct the interpolants $I(\cdot, \bfX, \bfs_0), I(\cdot, \bfX, \bfV), I(\cdot, \bfX, \bfW)$ 

   \State Define the decoder  $g_{\bfV,\bfW,I}$ 
 as in \eqref{eq:QMNG:ContEncoder}

    \State \textbf{return} $I(\cdot, \bfX, \bfs_0), I(\cdot, \bfX, \bfV), I(\cdot, \bfX, \bfW)$ and $g_{\bfV,\bfW,I}$ 
\EndProcedure
\end{algorithmic}
\end{algorithm}

\begin{algorithm}[t]
\caption{QMNG online phase}\label{alg:online phase}
\begin{algorithmic}[1]
\Procedure{OMNGOnline}{$\bfmu, \bfq_0(\bfmu), \delta t, K, I(\cdot, \bfX, \bfs_0), I(\cdot, \bfX, \bfV), I(\cdot, \bfX, \bfW),  g_{\bfV,\bfW,I}$}

    \State Encode the initial condition $\bfq_0(\bfmu)$ on quadratic manifold $\bftheta_0(\bfmu) = e_{\bfV}(\bfq_0(\bfmu))$ %
    \For{$k=1, \ldots, K$}
        \State Sample collocation points $\bfxi_1, \ldots, \bfxi_m \in \Omega$
        \State Assemble $\bfJ_{\Xi}(\bftheta_k(\bfmu))$ as in  \eqref{eq:QMNG:JTensor} %
            \State Assemble the batch right-hand side $\bff_{\Xi}(\bftheta_k(\bfmu); \bfmu)$ as in \eqref{eq:QMNG:BatchRHS}
            \State Compute $\delta{\bftheta}_k(\bfmu)$ as a solution of $\min_{\boldsymbol{\eta} \in \R^n} \norm{\bfJ_{\Xi}(\bftheta_k(\bfmu)) \boldsymbol{\eta} - \bff_{\Xi}(\bftheta_k(\bfmu); \bfmu)}_2^2$.
        \State Set $\bftheta_{k+1}(\bfmu) =\bftheta_k(\bfmu) + \delta t \delta{\bftheta}_k(\bfmu)$
    \EndFor
   \State \textbf{return} trajectory $\bftheta_0(\bfmu), \bftheta_1(\bfmu), \ldots, \bftheta_K(\bfmu)$
\EndProcedure
\end{algorithmic}
\end{algorithm}

\section{Neural Galerkin schemes with vector-valued quadratic  parametrizations}\label{sec:QMNGLinear}
We now consider the special case that the grid points $\bfx_1, \dots, \bfx_N$ of the full model \eqref{eq:Prelim:SemiDiscPDE} and the collocation points $\bfxi_1, \dots, \bfxi_m$ coincide. We derive an extension of the Neural Galerkin schemes that operates on vector-valued parametrizations to avoid having to introduce the interpolation operator as in Section~\ref{sec:QMNG}. This allows us to build on properties of the decoder functions corresponding to quadratic manifold to show the Jacobian matrices in this case remain full rank and the Neural Galerkin solution is locally unique. Furthermore, we show that for semi-discrete PDE problems \eqref{eq:Prelim:SemiDiscPDE} with right-hand sides that depend linearly on the state, Neural Galerkin schemes achieve online efficiency in the sense that the costs of time steps of QMNG reduced models in the online phase are independent of the dimension $N$ of the states of the full model \eqref{eq:Prelim:SemiDiscPDE}. 

\subsection{Neural Galerkin schemes with vector-valued quadratic manifold parametrizations}
The decoder function $g_{\bfV, \bfW}$ induces a vector-valued nonlinear parameterization 
\begin{equation}\label{eq:QMNG:VecNonLin}
\tilde{\bfq}: \Theta \to \mathbb{R}^N, \quad \tilde{\bfq}(\bftheta) = [\tilde{q}_1(\bftheta)\, \cdots\, \tilde{q}_N(\bftheta)]^{\top}\,.
\end{equation}
The component functions $\tilde{q}_1, \dots, \tilde{q}_N: \Theta \to \mathbb{R}$ correspond to point-wise approximations of the solution at the spatial coordinates $\bfx_1, \dots, \bfx_N$ at which the PDE problem \eqref{eq:Prelim:PDE} has been discretized to obtain the semi-discrete problem \eqref{eq:Prelim:SemiDiscPDE}; in other words the components functions $\tilde{q}_1, \dots, \tilde{q}_N$ are the component functions of the decoder $g_{\bfV,\bfW}$. 
Importantly, the vector-valued parametrization only provides approximations corresponding to the grid points $\bfx_1, \dots, \bfx_N$ and thus avoids having to perform an interpolation step as the parametrization $g_{\bfV,\bfW,I}$ used in Section~\ref{sec:QMNGLinear}. The weight vector is $\bftheta(t; \bfmu) \in \mathbb{R}^{n}$ with $n \ll N$. 

Analogous to the residual function \eqref{eq:Prelim:ResFun} for scalar-valued parametrizations, we now define a vector-valued residual function corresponding to the vector-valued parametrization $\tilde{\bfq}$ as
\begin{equation}\label{eq:QMNG:QMNGResFun}
\bfr: \Theta \times \Theta \to \mathbb{R}^N\,, \quad \bfr(\bftheta(t; \bfmu), \dot{\bftheta}(t; \bfmu))=\bfJ(\bftheta(t; \bfmu)) \dot{\bftheta}(t; \bfmu) - \bff(\tilde{\bfq}(\bftheta(t; \bfmu)); \bfmu)
\end{equation}
with the Jacobian matrix 
\begin{equation}\label{eq:QMNG:BatchGradient}
\bfJ(\bftheta(t; \bfmu)) = \begin{bmatrix}
\frac{\partial}{\partial \theta_1}\tilde{q}_1(\bftheta(t; \bfmu)) & \dots & \frac{\partial}{\partial \theta_p}\tilde{q}_1(\bftheta(t; \bfmu))\\
\vdots & & \vdots\\
\frac{\partial}{\partial \theta_1}\tilde{q}_n(\bftheta(t; \bfmu)) & \dots & \frac{\partial}{\partial \theta_p}\tilde{q}_n(\bftheta(t; \bfmu))
\end{bmatrix} \in \mathbb{R}^{N \times n}\,,
\end{equation}
and the right-hand side function $\bff$ of the full model \eqref{eq:Prelim:SemiDiscPDE}.
Notice that now the Jacobian matrix has size $N \times n$ and thus depends on the number of grid points $N$. 

Generalizing the Neural Galerkin conditions given in \eqref{eq:Prelim:ResEquation} to vector-valued parametrizations gives
\begin{equation}\label{eq:QMNG:NGCond}
\langle \partial_{\theta_i}\tilde{\bfq}(\bftheta(t; \bfmu)), \bfr(\bftheta(t; \bfmu), \dot{\bftheta}(t; \bfmu))\rangle_2 = 0\,,\qquad i = 1, \dots, n\,,
\end{equation}
which is formulated over the Euclidean inner product $\langle \cdot, \cdot \rangle_2$ so that the $L^2$ inner product over the spatial domain $\Omega$ used in \eqref{eq:Prelim:ResEquation} is numerically approximate via the mean Euclidean inner product over the spatial coordinates \eqref{eq:QMNG:SpatialCoordinates} of the semi-discrete problem \eqref{eq:Prelim:SemiDiscPDE}. 
Notice that the residual is set orthogonal to the space spanned by the columns of the Jacobian matrix $\bfJ(\bftheta(t; \bfmu))$.
The conditions \eqref{eq:QMNG:NGCond} correspond to the least-squares problem
\begin{equation}\label{eq:QMGNLinear:LSQProblem}
\min_{\dot{\bftheta}(t; \bfmu) \in \mathbb{R}^n} \|\bfJ(\bftheta(t; \bfmu))\dot{\bftheta}(t; \bfmu) - \bff(\tilde{\bfq}(\bftheta(t; \bfmu)); \bfmu)\|_2^2\,,
\end{equation}
to which we also refer in the following as vector-valued QMNG reduced model to distinguish it from the QMNG reduced model building on interpolation that is discussed in Section~\ref{sec:QMNG}.

We remark that in this special case of the nonlinear parametrization being a vector-valued decoder function corresponding to the grid points of the full model, the Neural Galerkin conditions coincide with the conditions used in \cite{LeeC2020Model} to define reduced dynamics; see \cite{zhang2024,BERMAN2024389} for other related schemes. %

\subsection{Full rankness of Jacobian matrix and uniqueness of Neural Galerkin solution}
It is a well-studied problem that Jacobian matrices of nonlinear parametrizations can become low rank so that the space against which the residual is set orthogonal in Neural Galerkin schemes collapses; see \cite{zhang2024} for an in-depth discussion. 
We now show that quadratic parametrizations that are trained with the greedy method introduced in \cite{SchwerdtnerP2024Greedy} avoid such a collapse if the Neural Galerkin schemes use the grid points as the collocation points. %
Recall that the greedy construction of quadratic manifolds yields matrices $\bfV$ and $\bfW$ that satisfy $\bfV^{\top}\bfW = 0$, which we use in the following lemma. Notice that this is a key difference to $\bfV_{\Xi}$ for which the analogous condition $\bfV_{\Xi}^{\top}\bfW_{\Xi} = 0$ can be violated. Modifying the greedy method to guarantee $\bfV_{\Xi}^{\top}\bfW_{\Xi} = 0$ is future work. %

We have the following result.
\begin{lemma}\label{lm:FullRank}
Let $\tilde{\bfq}$ be the nonlinear parametrization \eqref{eq:QMNG:VecNonLin} derived from a quadratic manifold constructed with the greedy approach discussed in Section~\ref{sec:Prelim:QM}. For all $\bftheta \in \mathbb{R}^n$, the Jacobian matrix $\bfJ(\bftheta)$ defined in \eqref{eq:QMNG:BatchGradient} of $\tilde{\bfq}$ has full column-rank $n$ and thus the least-squares problem~\eqref{eq:QMGNLinear:LSQProblem} has a unique solution.
\end{lemma}

\begin{proof}
Recall that $\bfV$ is a matrix with orthonormal columns and that $\bfV^{\top}\bfW = 0$ holds because we assume the quadratic manifold encoder $e_{\bfV}$ and decoder $g_{\bfV,\bfW}$ have been constructed with the greedy method introduced in \cite{SchwerdtnerP2024Greedy} and discussed in Section~\ref{sec:Prelim:QM}. 
Recall that $\bfJ(\bftheta)$ can be represented as \eqref{eq:QMNG:AffineJ}, with $\bfK$ defined analogously as $\bfK_{\Xi}$ in \eqref{eq:QMNG:KMapLinTensor}. Because $\bfK(\bftheta)$ consists of the multiplication of the matrix $\bfW$ with the matrix $\bftheta \otimes \bfI + \bfI \otimes \bftheta$, it holds that the column space of $\bfK(\bftheta)$ is a subspace of the column space of $\bfW$. 
Because $\bfV^{\top}\bfW = 0$ holds, the intersection of the column space spanned by $\bfV$ and $\bfW$ can only contain the zero vector. The matrix $\bfV$ has full column rank $n$, and thus the column space spanned by the $\bfV + \bfK(\bftheta)$ has to have dimension $n$, which shows that $\bfJ(\bftheta)$ has rank $n$.  The full rankness of $\bfJ(\theta)$ implies that the least-squares problem \eqref{eq:QMGNLinear:LSQProblem} is fully determined because $n < N$ and thus has a unique solution. 
\end{proof}

We now show that the Neural Galerkin conditions over quadratic manifolds lead to a well-posed problem that has a unique solution. %

\begin{proposition} %
    Assume that the right-hand side function $\bff(\cdot, \bfmu): \mathbb{R}^N \to \mathbb{R}^{N}$ of the semi-discrete problem  \eqref{eq:Prelim:SemiDiscPDE} is continuously differentiable  for all $\bfmu \in \Dcal$.  
    For a given $t_0 \in [0,T)$, $\bfmu \in \Dcal$, and initial condition $\bftheta_{t_0} \in \mathbb{R}^n$, there exists a unique trajectory $\bftheta(\cdot; \bfmu): [t_0, t_0+\varepsilon] \to \mathbb{R}^n$ that solves the Neural Galerkin conditions \eqref{eq:QMNG:NGCond}.  
\end{proposition}
\begin{proof}
We first note that the Neural Galerkin conditions \eqref{eq:QMNG:NGCond} are the normal equations corresponding to the least-squares problem \eqref{eq:QMGNLinear:LSQProblem} and thus we can write $\dot{\bftheta}(t; \bfmu)$ as
\begin{equation}\label{eq:QMNG:NEQProof}
\dot{\bftheta}(t; \bfmu) = \left(\bfJ(\bftheta(t; \bfmu))^{\top}\bfJ(\bftheta(t; \bfmu))\right)^{-1}\bfJ(\bftheta(t; \bfmu))^{\top}\bff(\tilde{\bfq}(\bftheta(t; \bfmu)); \bfmu)\,, \quad \bftheta(0; \bfmu)=\bftheta_{t_0}\,,
\end{equation}
where the inverse exists because the matrix $\bfJ(\bftheta(t; \bfmu))$ has full column rank; see Lemma~\ref{lm:FullRank}. 

We now show that the right-hand side of \eqref{eq:QMNG:NEQProof} is continuously differentiable in $\bftheta(t; \bfmu)$, which implies that it is locally Lipschitz. The map $\bftheta \mapsto \bfJ(\bftheta)$ is linear over finite-dimensional spaces and thus is continuously differentiable. Using that $\bfV^{\top}\bfW = \boldsymbol{0}$ (see proof of Lemma~\ref{lm:FullRank}), we obtain that $\bftheta \mapsto \bfJ(\bftheta)^{\top}\bfJ(\bftheta) = \bfI + \bfK(\bftheta)^{\top}\bfK(\bftheta)$, which is polynomial in $\bftheta$ and thus continuously differentiable. Because $\bfI + \bfK(\bftheta)^{\top}\bfK(\bftheta)$ is an invertible matrix for any $\bftheta \in \mathbb{R}^n$ (see Lemma~\ref{lm:FullRank}) and the matrix inverse is continuously differentiable over the set of invertible matrices in $\mathbb{R}^{n \times n}$, we obtain that $\bftheta \to (\bfI + \bfK(\bftheta)^{\top}\bfK(\bftheta))^{-1}$ is continuously differentiable over $\bftheta \in \mathbb{R}^n$.  This means that the right-hand side of \eqref{eq:QMNG:NEQProof} is a map in $\bftheta$ with domain $\mathbb{R}^n$ that is a composition of continuously differentiable functions and thus it is continuously differentiable as well. Thus, the right-hand side is locally Lipschitz and the Picard-Lindel\"of theorem is applicable, which shows the existence and uniqueness of a trajectory in an interval $[t_0, t_0 + \epsilon]$ for an $\epsilon > 0$. \end{proof}

\subsection{Online cost complexity and online efficiency for linear PDE models}\label{sec:QMNGLinear:Precompute}

Let us now consider the case of linear full models, which means that the function $\bff$ in the semi-discrete problem \eqref{eq:Prelim:SemiDiscPDE} depends linearly on the state vector $\bfq$. In particular, there exists a matrix $\bfA \in \mathbb{R}^{N \times N}$ such that  
\begin{equation}\label{eq:QMNG:OnlineCostCompAMatrix}
\bff(\bfq(t; \bfmu); \bfmu) =\bfA\bfq(t; \bfmu).
\end{equation}
Notice that we dropped the dependence on $\bfmu$ in the matrix $\bfA$. All of the following can be conducted for parametrized matrices $\bfA$ as long as the parameter dependence is affine or matrix interpolation is used \cite{RozzaPateraSurvey,SIREVSurvey}. However, for ease of exposition, we do not pursue this further here. 
If these conditions apply, we achieve online efficiency with QMNG reduced models if the collocation points coincide with the grid points, which means that the cost complexity of the online phase of the QMNG reduced models scales independently of the dimension $N$ of the states of the full model. 
To achieve online efficiency, we need to pre-compute several terms in the offline phase. First, we pre-compute the four-way tensor $\mathcal{J} \in \mathbb{R}^{n \times n \times n \times n}$ such that
\begin{equation}\label{eq:QMNG:MassTensor}
\bftheta \cdot \Jcal \cdot \bftheta = \bfJ(\bftheta)^T\bfJ(\bftheta) - \bfI \in \mathbb{R}^{n \times n}\,,
\end{equation}
with $\bfJ$ the Jacobian matrix evaluated at $\bftheta$ defined in \eqref{eq:QMNG:BatchGradient} and the tensor-vector multiplication that is analogously defined as for the three-way tensor in \eqref{eq:QMNG:TensorOne} and \eqref{eq:QMNG:TensorTwo}. Notice that $\bftheta \cdot \Jcal \cdot \bftheta = (\bftheta \cdot \Kcal)^T(\Kcal \cdot \bftheta)$, with $\Kcal$ given in \eqref{eq:QMNG:KMapLinTensor} but for $\Xi = \mathcal{X}$. However, we pre-compute $\Jcal$ to avoid having to perform a multiplication online that scales with the dimension $N$ of the semi-discrete problem \eqref{eq:Prelim:SemiDiscPDE}. 
We also pre-compute the vector $\bhats_0 \in \R^n$ and the matrices $\bhatA \in \R^{n \times n}$ and $\bhatH \in \R^{n \times n^2}$ as
\begin{align}\label{eq:Const proj operators}
    \bhats_0 = \bfV^\top \boldsymbol{s}_0 & &\bhatA = \bfV^\top \boldsymbol{A} \bfV & & \bhatH = \bfV^\top \bfA \bfW\,,
\end{align}
where $\bfs_0$ is the reference used in the decoder $g_{\bfV, \bfW}$ and $\bfV$ and $\bfW$ are the matrices of the decoder. The matrix $\bfA$ is determined by the right-hand side of the semi-discrete problem \eqref{eq:Prelim:SemiDiscPDE} as in \eqref{eq:QMNG:OnlineCostCompAMatrix}. 
Furthermore, we define the tensors $\mathcal{S} \in \R^{n \times n}, \mathcal{A} \in \R^{n \times n \times n}$, and $\mathcal{H} \in \R^{n \times n \times n^2}$ by
\begin{align}\label{eq:Lin proj operators}
    \mathcal{S}_{jk} = \mathcal{K}_{ikj} [\bfs_0]_i, & &
    \mathcal{A}_{kjm} = \mathcal{K}_{ijk} A_{il} V_{lm}, & &
    \mathcal{H}_{kjm} =\mathcal{K}_{ijk} A_{il} V_{lm}\,,
\end{align}
in Einstein notation. 
To achieve online efficiency, we solve the least-squares problem \eqref{eq:QMGNLinear:LSQProblem} via its normal equations 
\begin{equation}\label{eq:QMNG:CostNormalEq}
(\bfI + \bftheta(t; \bfmu) \cdot \Jcal \cdot \bftheta(t; \bfmu))\dot{\bftheta}(t; \bfmu) = \bhatf(\bftheta(t; \bfmu))\,,
\end{equation}
where the right-hand side is defined as 
\begin{align}\label{eq:PreCompRHS linear}
    \bhatf(\bftheta) =\bhats_0 + \bhatA \bftheta + \bhatH(\bftheta \otimes \bftheta)
    + (\bftheta \cdot \mathcal{S}) + (\bftheta \cdot \mathcal{A}) \bftheta  + (\bftheta \cdot \mathcal{H}) (\bftheta \otimes \bftheta)\,.
\end{align}
Using the pre-computed quantities, the left-hand side and the right-hand side of the normal equations \eqref{eq:QMNG:CostNormalEq} can be computed online with costs that scale independently of the dimension $N$ of the full model: Computing the solution $\dot{\bftheta}(t; \bfmu)$ incurs costs that scale as $\mathcal{O}(n^3)$ but independently of the dimension $N$ of the full model. The online assembly costs of the remaining terms in the reduced dynamics~\eqref{eq:QMNG:CostNormalEq} are listed in Table \ref{tab:QMNGCompCosts}. The most expensive term is $(\bftheta \cdot \mathcal{H})(\bftheta \otimes \bftheta)$. Indeed, the cost of the operation $(\bftheta \cdot \mathcal{H})$ scales as $\mathcal{O}(n^4)$.
The computational cost of one time step therefore scales as $\mathcal{O}(n^4)$.

\begin{table}[]
\centering
\caption{Cost complexities of assembling the terms used in~\eqref{eq:QMNG:CostNormalEq} and~\eqref{eq:PreCompRHS linear}. QMNG reduced models for linear problems incur online costs that scale at most as $\mathcal{O}(n^4)$, which is independent of the dimension $N$ of the full model.}
\label{tab:QMNGCompCosts}
\begin{tabular}{@{}ll@{}}
\toprule
\textbf{term}                        & \textbf{cost complexity} \\ \midrule
$\bftheta \cdot \mathcal{J} \cdot \bftheta$               & $\mathcal{O}(n^4) + \mathcal{O}(n^3)$ \\
$\bhatA \bftheta$                    & $\mathcal{O}(n^2)$          \\
$\bhatH (\bftheta \otimes \bftheta)$ & $\mathcal{O}(n^2) + \mathcal{O}(n^3)$          \\
$\bftheta \cdot \mathcal{S}$         & $\mathcal{O}(n^2)$          \\
$(\bftheta \cdot \mathcal{A}) \bftheta$                   & $\mathcal{O}(n^3) + \mathcal{O}(n^2)$ \\
$(\bftheta \cdot \mathcal{H})(\bftheta \otimes \bftheta)$ & $\mathcal{O}(n^2) + \mathcal{O}(n^4) + \mathcal{O}(n^3)$ \\ \bottomrule
\end{tabular}
\end{table}

\subsection{Algorithms for online-efficient QMNG reduced models of linear PDE problems}
Algorithm~\ref{alg:linear offline phase} summarizes the offline phase for constructing QMNG reduced models for linear problems. The inputs to the offline phase are the snapshot matrix $\bfQ$, the system matrix $\bfA$ given in \eqref{eq:QMNG:OnlineCostCompAMatrix} of the full model, the reduced dimension $n$, the regularization parameter $\gamma$, and the number of candidate singular vectors $l$; for details on the last two parameters, we refer to \cite{SchwerdtnerP2024Greedy}. The vector $\bfs_0$ and the matrices $\bfV$ and $\bfW$ are computed with the greedy method in the same way as in Algorithm \ref{alg:offline phase}. The tensors $\mathcal{K}$ and $\mathcal{J}$ are computed following \eqref{eq:QMNG:KMapLinTensor} and \eqref{eq:QMNG:MassTensor}, respectively. The vectors and matrices $\bhats_0, \bhatA$, and $\bhatH$ are computed following \eqref{eq:Const proj operators}. Finally, the tensors $\mathcal{S}, \mathcal{A}$, and $\mathcal{H}$ are computed as in \eqref{eq:Lin proj operators}. The quantities are returned.

Algorithm~\ref{alg:linear online phase} summarizes the steps of the online phase QMNG reduced models with forward Euler time integration for linear problems. Analogously to Algorithm~\ref{alg:online phase}, other time integration schemes can be used. The inputs to the algorithm are the parameter vector $\bfmu$, the corresponding initial condition $\bfq_0(\bfmu)$, the time-step size $\delta t$ and the number of time steps $K$. Additionally, the quantities pre-computed in the offline phase are inputs too. 
The weight trajectory is initialized by encoding the initial condition $\bfq_0(\bfmu)$ into the quadratic manifold as $\bftheta_0(\bfmu)$. At each time step, we only have to assemble the matrix $\bfI + \bftheta_k(\bfmu) \cdot \mathcal{J} \cdot \bftheta_k(\bfmu)$ and the reduced right-hand side  $\bhatf(\btheta_k(\bfmu))$. Then, $\dot{\bftheta}_k(\bfmu)$ is obtained by solving the linear system~\eqref{eq:QMNG:CostNormalEq} of size $n\times n$ using a Cholesky decomposition and the parameter vector is updated. Indeed, due to the structure of $\mathcal{J}$, the matrix $\bfI + \bftheta_k(\bfmu) \cdot \mathcal{J} \cdot \bftheta_k(\bfmu)$ is always symmetric positive definite. The algorithm returns the trajectory of the reduced parameters at the time steps $t_0, \ldots, t_K$.

\begin{algorithm}[t]
\caption{Offline phase of QMNG reduced models for linear problems}\label{alg:linear offline phase}
\begin{algorithmic}[1]
\Procedure{QMNGOfflineLinear}{$\bfQ, \bfA, n, \gamma, l$}

   \State Apply the greedy method to $\bfQ$ with parameters $n, \gamma, \ell$ to obtain $\bfs_0, \bfV$ and $\bfW$.

   \State Assemble $\mathcal{K}$ according to~\eqref{eq:QMNG:KMapLinTensor}
        \State Assemble $\mathcal{J}$ according to ~\eqref{eq:QMNG:MassTensor}
        \State Assemble $\bhats_0, \bhatA, \bhatH$ according to ~\eqref{eq:Const proj operators}
        \State Assemble $\mathcal{S}, \mathcal{A}, \mathcal{H}$ according to ~\eqref{eq:Lin proj operators}
        \State \textbf{return} $\bfV, \bfW, \bhats_0, \bhatA, \bhatH, \mathcal{J}, \mathcal{S}, \mathcal{A}, \mathcal{H}$
    
\EndProcedure
\end{algorithmic}
\end{algorithm}

\begin{algorithm}[t]
\caption{Online phase of QMNG reduced models for linear problems}\label{alg:linear online phase}
\begin{algorithmic}[1]
\Procedure{QMNGOnlineLinear}{$\bfmu, \bfq_0(\bfmu), \delta t, K, \bfs_0, \bfV, \mathcal{J}, \bhats_0, \bhatA, \bhatH, \mathcal{S}, \mathcal{A},  \mathcal{H}$}

    \State Encode the initial condition $\bfq_0(\bfmu)$ on quadratic manifold $\bftheta_0(\bfmu) = e_{\bfV}(\bfq_0(\bfmu))$ 
    \For{$k=1, \ldots, K$}
            \State Assemble $\bfI + \bftheta_k(\bfmu) \cdot \mathcal{J} \cdot \bftheta_k(\bfmu)$
            \State Compute $\bhatf(\btheta_k(\bfmu))$ in~\eqref{eq:PreCompRHS linear}
            \State Compute $\delta{\bftheta}_k(\bfmu)$ as the solution of $(\bfI + \bftheta_k(\bfmu) \cdot \mathcal{J} \cdot \bftheta_k(\bfmu)) \delta{\bftheta}_k(\bfmu) = \bhatf(\btheta_k(\bfmu))$
            \State Set $\bftheta_{k+1}(\bfmu) =\bftheta_k(\bfmu) + \delta t \delta{\bftheta}_k(\bfmu)$
    \EndFor
   \State \textbf{return} trajectory $\bftheta_0(\bfmu), \bftheta_1(\bfmu), \ldots, \bftheta_K(\bfmu)$
\EndProcedure
\end{algorithmic}
\end{algorithm}

\section{Numerical experiments}\label{sec:NumExp}
We now demonstrate QMNG reduced models on a range of numerical problems. We first consider QMNG that directly uses the vector-valued quadratic decoder function as nonlinear parametrization and demonstrate runtime speedups and online efficiency when simulating acoustic waves in a two-dimensional spatial domain in Section~\ref{sec:NumExp:Wave} and charged particles in Section~\ref{sec:NumExp:Vlasov}. We then consider QMNG with interpolated decoder functions and apply it to the Burgers' equation in Section~\ref{sec:NumExp:Burgers}, where we show that the number of collocation points can be chosen smaller than the number of grid points used in the full model. As we argued above, separating collocation points in Neural Galerkin schemes from the grid points of the full model means that no empirical interpolation or hyper-reduction is needed for problems with nonlinear dynamics.

\subsection{Hamiltonian wave}\label{sec:NumExp:Wave}
We consider acoustic waves traveling in a two-dimensional spatial domain. 

\subsubsection{Setup}
Let the two-dimensional spatial domain be $\Omega = [-4, 4)^2$ and consider the acoustic wave equation in Hamiltonian form 
\begin{equation}
\begin{aligned}
    \partial_t \rho( \bfx, t; \mu) &= - \nabla \cdot \boldsymbol{v}(\bfx, t; \mu),\\
    \partial_t \boldsymbol{v}(\bfx, t; \mu) &= -\nabla \rho(\bfx, t;\mu),\\
    \rho(\bfx, 0; \mu) &= \rho_0(\bfx;\mu),\\
    \boldsymbol{v}(\bfx, 0) &= \boldsymbol{0},
\end{aligned}
\end{equation}
where $\rho: \Omega \times [0, T] \times \Dcal \to \mathbb{R}$ and $\bfv: \Omega \times [0, T] \times \Dcal \to \mathbb{R}^2$ denote the density and velocity fields, respectively; see  \cite[Appendix A.2]{SchwerdtnerP2024Greedy} for more details. The physics parameters are in $\mu \in \Dcal = [0, 1]$ and the time interval is $[0, 8]$. 
We impose periodic boundary conditions and consider for the density field the initial condition
\[
    \rho_0(\bx; \mu) = \exp(-(6 + \mu)^{2} ||\bfx - \bfx_0||_2^2)\,,
\]
with $\bfx_0 = [2, 2]^\top$. The initial velocities are set to zero. The full model is obtained with a central finite-difference scheme in space. We use 1024 grid points per spatial dimension, which results in a total full-model dimension of $N = 3,145,728$ taking into account that we consider three quantities. In time, we discretize with fourth-order Runge-Kutta and a time-step size $10^{-3}$. 

\begin{table}[]
\caption{Test parameters used for computing the averaged relative errors of the QMNG reduced models.}
\label{tab:test parameters}
\centering
\begin{tabular}{@{}ll@{}}
\toprule
\textbf{problem} & \textbf{test  physics parameters}                \\ \midrule
acoustic waves  & 0.0556, 0.3889, 0.5000, 0.7222, 0.9444  \\
charged particles            & 0.2765,  0.2724, 0.3500, 0.3827, 0.4398 \\
Burgers' equation & 0.3810, 0.3837, 0.5000, 0.5490, 0.6347\\
\bottomrule
\end{tabular}
\end{table}

For training  quadratic manifolds, we generate snapshot data with the full model following the procedure described in Section~\ref{sec:QMNG:Snapshots}: We generate ten trajectories corresponding to the $M' = 10$ equidistant training parameters $\mu_1^{\text{(train)}}, \dots, \mu_{M'}^{\text{(train)}}$ in $\Dcal$. We subsample the ten training trajectories in time by a factor 40 so that one training trajectory contains 200 snapshots in time. In total, there are 2,000 snapshots, which form the snapshot matrix \eqref{eq:QMNG:SnapshotQ}. The reference point $\bfs_0$ is the mean over all snapshots. We also generate five test trajectories corresponding to the five test parameters $\mu_1^{\text{(test)}}, \dots, \mu_5^{\text{(test)}}$ shown in Table~\ref{tab:test parameters}. Notice that the test parameters have a large distance from the $M' = 10$ equidistant training parameters. We remark that our approach does not necessarily preserve the Hamiltonian corresponding to this problem over time. Combining the proposed QMNG reduced models with the nonlinear embeddings introduced in \cite{doi:10.1137/23M1607799} to preserve quantities such as energy remains future work. 

\begin{figure}
    \centering
    \includegraphics[width=\linewidth]{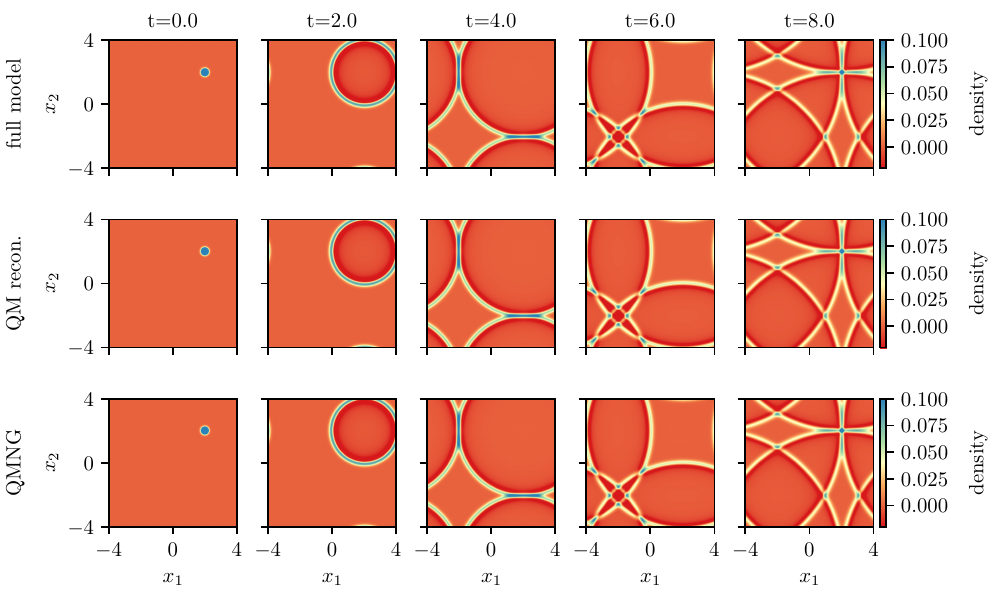}
    \caption{Acoustic waves: Density field snapshots from the full model of the Hamiltonian wave problem (first row), their approximations on the quadratic manifold (second row) and the corresponding approximation obtained with the QMNG reduced model of dimension $n=30$ (third row). Results are shown for the test parameter $\mu_3^{\text{(test)}} = 0.5$.}
    \label{fig:hwave snapshots}
\end{figure}

\subsubsection{Results}
We now take the snapshot data and train quadratic manifolds of dimension $n \in \{5, 10, \dots, 50\}$ with the greedy method introduced in \cite{SchwerdtnerP2024Greedy} and as discussed in Section~\ref{sec:Prelim:QM}. The number of candidate singular vectors is $200$. The regularization parameter $\gamma$ used in fitting the weight matrix $\bfW$ in \eqref{eq:Prelim:FitW} is manually selected based on training data and set to $\gamma = 10^{-6}$.
Let us consider Figure~\ref{fig:hwave snapshots} now. The first row shows the density field of the full-model solution over times $t \in \{0, 2, 4, 6, 8\}$ for test parameter $\mu_3^{\text{(test)}} = 0.5$. The second row shows the reconstruction on the quadratic manifold of dimension $n = 30$, which agrees well with the full-model solution. The construction of a full-model solution $\bfq(t; \bfmu) \in \mathbb{R}^N$ on the quadratic manifold is obtained as $\hat{\bfq}(t; \bfmu) = g_{\bfV,\bfW}(e_{\bfV}(\bfq(t; \bfmu))$, where $e_{\bfV}$ is the encoder function and $g_{\bfV,\bfW}$ is the decoder function of the quadratic manifold. Let us now consider the QMNG reduced model \eqref{eq:QMGNLinear:LSQProblem} when the collocation points coincide with the grid points of the full model, $\bfxi_1 = \bfx_1, \dots, \bfxi_N = \bfx_N$, which is the case considered in Section~\ref{sec:QMNGLinear}. This means that the nonlinear parameterization is the vector-valued quadratic decoder function $g_{\bfV,\bfW}$. Because the acoustic wave equation is linear in the solution variables, we can pre-compute quantities as discussed in Section~\ref{sec:QMNGLinear:Precompute} and achieve online efficiency, which means that the online costs of the QMNG reduced model scales independently of the dimension $N$ of the full model. The QMNG reduced model constructs a weight trajectory $\bftheta(t_1; \mu), \dots, \bftheta(t_K; \mu)$ for the test parameter $\mu_3^{(\text{test})} \in \Dcal$, which then gives rise to the approximations $\tilde{\bfq}(t_1; \mu) = g_{\bfV,\bfW}(\bftheta(t_1; \mu)), \dots, \tilde{\bfq}(t_K; \mu) = g_{\bfV,\bfW}(\bftheta(t_K; \mu))$, which are plotted in the third row of Figure~\ref{fig:hwave snapshots} for $t \in \{0, 2, 4, 6. 8\}$. The QMNG reduced solutions are in close agreement with the full-model solutions.

\begin{figure}
    \centering
        \begin{subfigure}[T]{0.48\textwidth}
        \includegraphics{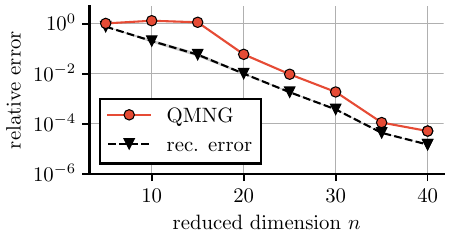}
    \caption{relative error}
    \end{subfigure}
    \hfill
    \begin{subfigure}[T]{0.48\textwidth}
        \includegraphics{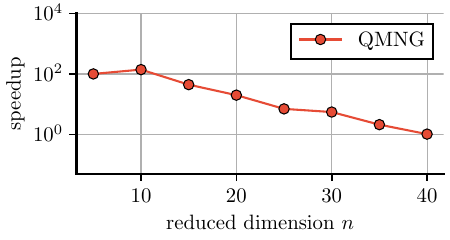}
    \caption{speedup}
    \end{subfigure}

    \caption{Acoustic waves: Averaged relative test error of QMNG reduced models for different reduced dimensions $n$ and the corresponding online speedups. The QMNG reduced model achieves an error that is close to the reconstruction error of the data on the quadratic manifold, which shows that QMNG leverages the expressivity of the quadratic manifold.} %
    \label{fig:Wave:QMNG:Error}
\end{figure}

To assess the accuracy of QMNG reduced solutions in more detail, we consider the averaged relative error\begin{equation}\label{eq:NumExp:RelError}
\mathcal{E} = \frac{1}{5}\sum_{i = 1}^5 \sum_{k = 1}^K \frac{\|\bfq(t_k; \mu_i^{\text{(test)}}) - \breve{\bfq}(t_k; \mu_i^{\text{(test)}})\|_2}{\|\bfq(t_k; \mu_i^{\text{(test)}})\|_2}\,,
\end{equation}
where $\breve{\bfq}(t_k; \mu_i^{\text{(test)}})$ is an approximation of the full-model solution $\bfq(t_k; \mu_i^{\text{(test)}})$ such as  the approximation $\hat{\bfq}(t_k; \mu_i^{\text{(test)}})$ obtained by reconstructing the full-model solution  on the quadratic manifold or the QMNG reduced solution $\tilde{\bfq}(t_k; \mu_i^{\text{(test)}})$. We plot the averaged relative error \eqref{eq:NumExp:RelError} in Figure~\ref{fig:Wave:QMNG:Error}a and the speedup obtained with the QMNG reduced model in Figure~\ref{fig:Wave:QMNG:Error}b over the reduced dimensions $n \in \{5, 10, \dots, 50\}$. The shaded area in Figure~\ref{fig:Wave:QMNG:Error}a corresponds to the standard deviation of the averaged relative error $\mathcal{E}$ over the test parameters. The error of the QMNG reduced solution decays similarly to the reconstruction error, which shows that QMNG reduced models leverage the expressivity of the quadratic manifolds. The speedup plot shows that the QMNG reduced model achieves orders of magnitude speedups for reduced dimensions $n \leq 30$. %
For dimensions $n \geq 40$, QMNG reduced models have a higher runtime than the full model in this example, which is explained by the online cost complexity that scales as $\mathcal{O}(n^4)$ in the reduced dimension $n$, even though it scales independently of the full-model dimension $N$.

\begin{figure}
    \centering
    \begin{subfigure}[T]{0.48\textwidth}
    \includegraphics{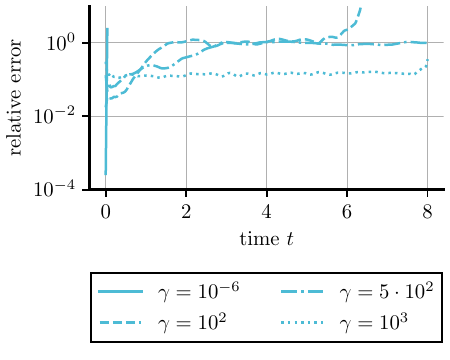}
    \caption{QMNG, constant-in-time test space}
    \end{subfigure}
    \hfill
    \begin{subfigure}[T]{0.48\textwidth}
    \includegraphics{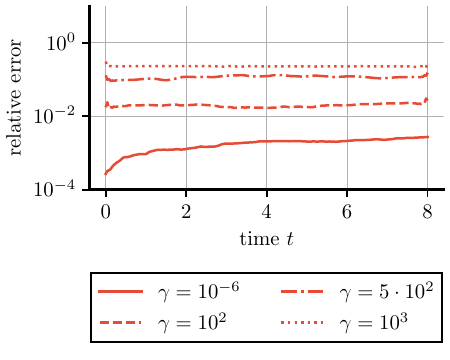}
    \caption{QMNG}
    \end{subfigure}
    \caption{Acoustic waves: Reduced models that use constant approximations of the Jacobian matrix are unstable for quadratic manifolds that are fitted well to the data (low regularization parameter $\gamma$). In contrast, QMNG reduced models use the actual Jacobian matrix without an approximation and so minimize the residual norm, which leads to stable predictions also when quadratic manifolds are trained with small regularization parameters.}
    \label{fig:hwave relerrs over time}
\end{figure}

\begin{figure}
    \centering
        \begin{subfigure}[T]{0.48\textwidth}
        \includegraphics{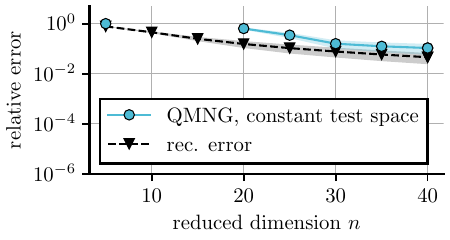}
    \caption{averaged relative error}
    \end{subfigure}
    \hfill
    \begin{subfigure}[T]{0.48\textwidth}
        \includegraphics{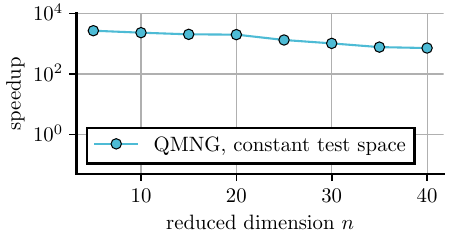}
    \caption{speedup}
    \end{subfigure}

    \caption{Acoustic waves: The reduced dynamics corresponding to constant approximations of the Jacobian matrix are unstable in our experiments and require strongly regularized quadratic manifolds to be stable (here $\gamma = 10^3$). At the same time, using constant test spaces avoids having to assemble the Jacobian matrix as in QMNG reduced models, which leads to higher speedups compared to QMNG.} %
    
    \label{fig:Wave:QMConst:Error}
\end{figure}

\subsubsection{Comparison to reduced models with constant-in-time test spaces}
We discussed in Section~\ref{sec:QMNG:ResNormMin} that QMNG reduced models minimize the residual norm by setting the residual orthogonal to the space spanned by the columns of the Jacobian matrix of the quadratic decoder function. We now study the importance of this test space. For this, we consider the zeroth-order approximation $\bar{\bfJ}$ that leads to the reduced dynamics \eqref{eq:QMNG:OrthoDyn}. We show in Figure~\ref{fig:hwave relerrs over time} the error that is achieved with QMNG reduced models that minimize the norm of the residual and the analogous reduced models \eqref{eq:QMNG:OrthoDyn} that set the residual orthogonal to a constant-in-time test space, which does not necessarily lead to residual minimization. To obtain approximations with the constant-in-time test spaces, we need to use stronger regularization when constructing the manifold:  For regularization parameters $\gamma \in \{10^{-6}, 10^2, 5 \times 10^2, 10^3\}$, the QMNG reduced models achieve stable behavior, even though the error increases as the regularization parameter $\gamma$ is increased because the quadratic manifold is less well fit to the data as $\gamma$ increases. In contrast, the analogous reduced model that sets the residual orthogonal to the constant approximation of the Jacobian is unstable for quadratic manifolds trained with small regularization parameters, i.e., when the quadratic manifold is well fit to data. Only when the regularization parameter is increased, and correspondingly the quadratic manifold poorer approximates the snapshots, the reduced model provides stable predictions; however the predictions have larger errors corresponding to the poorer quadratic manifold. Figure~\ref{fig:Wave:QMConst:Error} report the error decay for reduced models with constant-in-time test spaces, which show less stable behavior and achieve orders of magnitude higher errors than QMNG reduced models because the regularization parameter has to be so large. At the same time, the runtime of the reduced models based on the constant-in-time test spaces is lower than the runtime of the QMNG reduced model because assembling the Jacobian matrix incurs higher costs. 

The results in Figure~\ref{fig:hwave relerrs over time} and Figure~\ref{fig:Wave:QMConst:Error} indicate that minimizing the residual norm as QMNG reduced models is key for stable predictions.

\subsection{Charged particles in fixed potential}\label{sec:NumExp:Vlasov}
We now consider charged particles, which are governed by the Vlasov equation with a fixed electrostatic potential.

\subsubsection{Setup}
We consider the Vlasov equation
\begin{equation}\label{eq:NumExp:Vlasov:PDE}
\begin{aligned}
    \partial_t q(x, v, t; \mu) &= -v \partial_x q(x, v, t; \mu) + \partial_x \phi(x; \mu) \partial_v q(x, v, t; \mu),\\
    q(x, v, 0; \mu) &= q_0(x, v),
\end{aligned}
\end{equation}
where $x \in \Omega_x = [-1, 1)$ denotes the spatial coordinate and $v \in \Omega_v = [-1, 1)$ the velocity coordinate. The function $\phi: \Omega_x \to \R$ denotes an electrostatic potential, which is constant in time in our setup. We impose periodic boundary conditions and solve the problem over the time interval $[0, 3.2]$. The initial condition $q_0$ is set to
\begin{align}
    q_0(x, v) = \frac{1}{2 \pi \sigma} \exp \left( - \frac{1}{\pi \sigma}
    \left[\sin(\frac{\pi}{2}(x - x_0))^2 + \sin(\frac{\pi}{2}(v - v_0))^2 \right] \right),
\end{align}
where $\sigma = 8 \cdot 10^{-3}$, $x_0 = -0.2$, and $v_0 = 0$ are fixed. The problem is parameterized via the potential function $\phi$, 
\begin{align}
    \phi(x; \mu) = -\alpha (1 + \cos(\pi(x + \mu))^4) - \beta \sin(\pi x), 
\end{align}
where $\mu \in [0.25, 0.45]$, $\alpha = 0.2$ and $\beta = 0.1$ are fixed. The full model is obtained with a fourth-order central finite difference scheme in space and fourth-order Runge-Kutta scheme in time. We use 512 equidistantly spaced grid points in each dimension, which leads to the dimension $N = 262,144$ of the full-model states $\bfq(t; \mu)$. The time-step size is $10^{-3}$.

We generate snapshot data with the full model by selecting $M' = 50$ training physics parameter $\mu_1, \dots, \mu_{M'}$ that are equidistantly spaced in $\Dcal$. The training trajectories are subsampled by a factor ten in time so that each trajectory contains 320 snapshots. In total, there are 16,000 snapshots that form the snapshot matrix \eqref{eq:QMNG:SnapshotQ}, which is used for training the encoder and decoder functions of quadratic manifolds. Notice that we use five times more training physics parameters compared to the acoustic wave problem, which reflects that this problem is more challenging for quadratic manifolds. The reference point $\bfs_0$ is the mean over all training snapshots. As in the previous example, we generate five test trajectories corresponding to five test physics parameters, which are listed in Table \ref{tab:test parameters}. Each of the test physics parameters is a midpoint between two neighboring training physics parameters.

\subsubsection{Results}

\begin{figure}
    \centering
    \includegraphics[width=\linewidth]{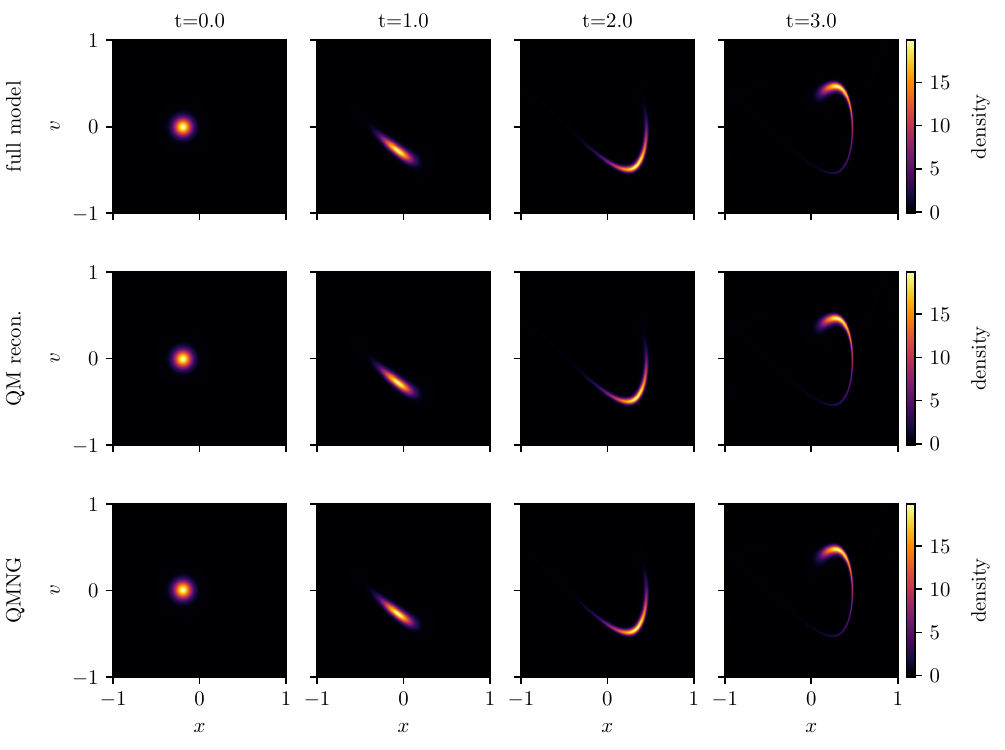}
    \caption{Charged particles: Snapshots from the full model of the Vlasov problem, their reconstruction after projection onto the quadratic manifold and the corresponding reduced solutions from the QMMG reduced model of size $n = 30$.} %
    \label{fig:vlasov snapshots}
\end{figure}

\begin{figure}
    \centering
        \begin{subfigure}[T]{0.48\textwidth}
        \includegraphics{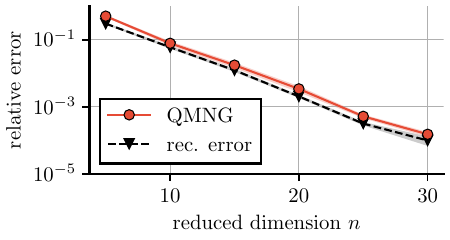}
    \caption{averaged relative error on test parameters}
    \end{subfigure}
    \hfill
    \begin{subfigure}[T]{0.48\textwidth}
        \includegraphics{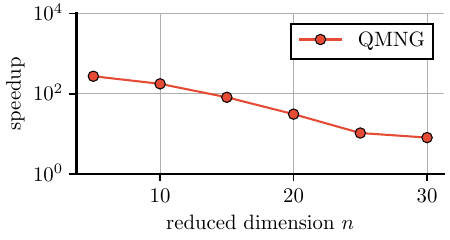}
    \caption{speedups}
    \end{subfigure}

    \caption{Charged particles: The QMNG reduced model leverages the expressivity of quadratic manifolds to achieve a similar error decay as the reconstruction error. At the same time, up to two orders of magnitude speedups are achieved compared to the full model.} %
    \label{fig:Vlasov:QMNG:Error}
\end{figure}

\begin{figure}
    \centering
    \begin{subfigure}[T]{0.48\textwidth}
    \includegraphics{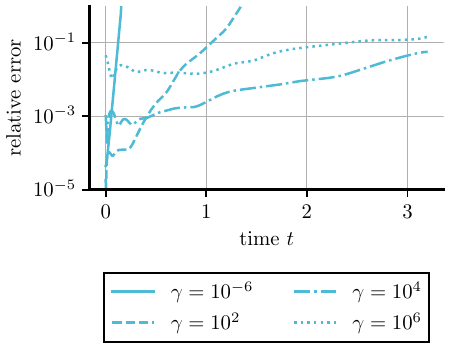}
    \caption{QMNG, constant-in-time test space}
    \end{subfigure}
    \hfill
    \begin{subfigure}[T]{0.48\textwidth}
    \includegraphics{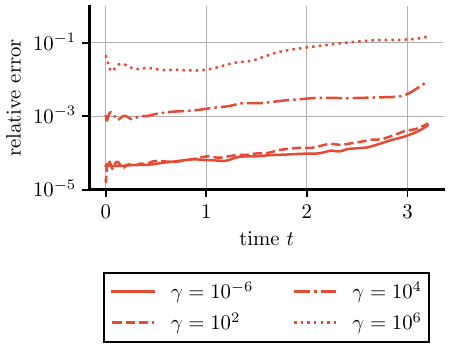}
    \caption{QMNG}
    \end{subfigure}
    \caption{Charged particles: Using a constant-in-time test space (zeroth-order approximation of Jacobian) leads to unstable predictions for small regularization parameters when the quadratic manifold is fitted well to data. In contrast, the proposed QMNG reduced models use a test space that varies with time and that corresponds to residual minimization, which provides stabler predictions also when the quadratic manifolds are fitted well to data.}
    \label{fig:vlasov relerrs over time}
\end{figure}

We use the generated snapshot data to train quadratic manifolds with dimension $n \in \{5, 10, \dots, 30\}$. The regularization parameter has been chosen as $\gamma = 10^{-6}$. For times $t \in \{0, 1, 2, 3\}$, we compare the full-model solution at test physics parameter $\mu_3^{\text{(test)}} = 0.35$ with the reconstruction on the quadratic manifold of dimension $n = 30$ in Figure~\ref{fig:vlasov snapshots}. The approximation on the quadratic manifold is in good agreement with the full-model solution. Let us now consider the corresponding QMNG reduced model, which is based on the vector-valued decoder function as nonlinear parametrization analogous to the setup of  the previous example discussed in Section~\ref{sec:NumExp:Wave}. The Vlasov equation \eqref{eq:NumExp:Vlasov:PDE} is linear in the solution variable $q$ and thus we can pre-compute the terms described in Section~\ref{sec:QMNGLinear:Precompute} to achieve online efficiency with QMNG reduced models in this example. The third row of Figure~\ref{fig:vlasov snapshots} shows that the QMNG reduced solution at the test parameter $\mu_i^{\text{(test)}} = 0.35$ is in close agreement with the full-model solution.

In Figure~\ref{fig:Vlasov:QMNG:Error}a, we plot the averaged relative error \eqref{eq:NumExp:RelError} over the test physics parameter for the QMNG reduced model as well as the reconstruction error of the test trajectories on the quadratic manifold. The shaded area corresponds to the standard deviation of the relative error over the test parameters. The error of the QMNG reduced solution decays similarly as the reconstruction error, which shows that QMNG can leverage well the expressivity the quadratic manifold. Figure~\ref{fig:Vlasov:QMNG:Error}b shows the speedup, which is multiple orders of magnitude compared to the full model in this example.

Let us now demonstrate the importance of minimizing the residual norm by considering the analogous reduced model that sets the residual orthogonal to the constant test space spanned by the zeroth-order approximation of the Jacobian. Thus, the residual norm is not minimized in this case; see Section~\ref{sec:QMNG:ResNormMin}. Analogous to the results for the acoustic wave equation, the reduced model based on the constant-in-time test space requires a quadratic manifold with a large regularization parameter $\gamma$ to provide stable approximations; see Figure~\ref{fig:vlasov relerrs over time}. Because the quadratic manifold has to be trained with a large regularization parameter,  the full potential of quadratic manifolds cannot be leveraged because a larger regularization parameter prevents fitting the quadratic manifold to high accuracy. This is in contrast to the QMNG reduced model, which provides accurate and stable approximations also for small regularization parameters so that the quadratic manifold is fit well to the snapshot data.  Correspondingly, the QMNG reduced model can achieve orders of magnitude lower errors. %

\subsection{Burger's equation}\label{sec:NumExp:Burgers}
We consider Burger's equation over a one-dimensional spatial domain, which is nonlinear in the solution variable. 
\subsubsection{Setup} Let $\Omega = [-1, 1)$ be the spatial domain and consider Burgers' equation
\begin{align}
    \partial_t q(x, t; \mu) &= q(x, t; \mu) \partial_x q(x, t; \mu) + \alpha \partial_{xx} q(x, t; \mu),\\
    q(x, 0; \mu) &= q_0(x; \mu),
\end{align}
with the viscosity $\alpha$ set to $0.005$, the physics parameters $\mu \in \Dcal = [0.35, 0.65]$, and the time interval $[0, 1]$. 
We impose periodic boundary conditions and consider the initial condition
\begin{align}
    q(x; \mu) = \frac{1}{2 \pi \sigma} \exp \left( - \frac{1}{\pi \sigma} \left|\sin(\frac{\pi}{2}(x - \mu))\right|^2\right),
\end{align}
with $\sigma = 0.005$. The full model is obtained with a finite difference scheme in space, where we use a fourth-order central finite difference stencil for the first-order derivative and a second-order central finite difference stencil for the second-order derivative. We use $N = 2048$ equidistant grid points in space. In time, we discretize with a time-step size of $10^{-4}$ and the  fourth-order explicit Runge-Kutta method.

We generate snapshot data for training the quadratic manifold encoder and decoder functions. To do this, we select the $M' = 50$ equidistant physics parameter $\mu_1, \dots, \mu_{M'}$ in the domain $\Dcal$ and construct the corresponding 50 time trajectories. We subsample the trajectories in time by a factor 50 such that each training trajectory consists of $200$ snapshots in time and the quadratic manifolds are trained on $10,000$ snapshots in total, i.e., the number of columns of the snapshot matrix $\bfQ$ given in \eqref{eq:QMNG:SnapshotQ} is $10,000$. The reference point $\bfs_0$ is the mean over all snapshots. To evaluate the performance of the reduced models, we generate five  test trajectories that correspond to five physics parameters  that are chosen at the midpoints between the training parameters in $\Dcal$ and that are listed in Table \ref{tab:test parameters}.

\subsubsection{Results} We now take the snapshot data and train quadratic manifolds of dimension $n \in \{5, 10, ..., 50\}$ with regularization parameter $\gamma = 10^{-6}$ in \eqref{eq:Prelim:FitW}, which has been found via manual parameter tuning based on the training data. The first row of Figure~\ref{fig:BurgersSnapshots} shows the 
numerical solution field corresponding to the test parameter $\mu_3^{\text{(test)}} = 0.5$ for times $t \in \{0, 0.2, 0.5, 0.8, 1.0\}$. The second row of Figure~\ref{fig:BurgersSnapshots} shows the reconstruction of the full-model solution on the quadratic manifold of dimension $n = 30$, i.e., taking a full-model solution $\bfq(t; \mu)$ and computing the reconstruction $\hat{\bfq}(t; \mu) = g_{\bfV,\bfW}(e_{\bfV}(\bfq(t; \mu))$. Comparing the first and second row of Figure~\ref{fig:BurgersSnapshots} shows that the quadratic manifold can well approximate the snapshots in this example. Let us now consider the third row of Figure~\ref{fig:BurgersSnapshots}, which shows the approximation obtained with the QMNG reduced model that uses interpolation on 1024 uniformly sampled collocation points $\bfxi_1, \dots, \bfxi_{1024}$, which are resampled at each time step. Recall that the QMNG reduced model \eqref{eq:QMNG:NGXi} computes a trajectory of weights $\bftheta(t_1; \bfmu), \dots, \bftheta(t_K, \bfmu) \in \mathbb{R}^n$, which lead to the approximations 
\[
\tilde{\bfq}(t_i; \mu) = [g_{\bfV,\bfW,I}(\bfx_1, \bftheta(t_i; \mu)\, \cdots\, g_{\bfV,\bfW,I}(\bfx_N, \bftheta(t_i; \mu))]^{\top} \in \mathbb{R}^N
\]
where $\bfx_1, \dots, \bfx_N$ are the grid points corresponding to the full model and $\mu$ is a test physics parameter. Comparing the first and third row of Figure~\ref{fig:BurgersSnapshots} shows that the QMNG reduced model accurately predicts the full-model solutions for this test parameter.

Let us now consider the averaged relative error \eqref{eq:NumExp:RelError} over all test parameters. 
The averaged relative errors corresponding to QMNG reduced models over a range of reduced dimensions are shown in Figure~\ref{fig:BurgersErrors}a. The shaded area corresponds to the standard deviation of the relative errors over the test parameters. If the collocation points $\bfxi_1, \dots, \bfxi_m$ are selected equidistantly in the spatial domain and there are $m = 1024$ of them, which is half of the number of grid points $N = 2048$ used for the full model, then the error decays rapidly until about dimension $n = 40$, before it levels off. The same leveling off can be seen when using $m = 1024$ collocation points that are uniform randomly chosen in the spatial domain as well as in the reconstruction error on the quadratic manifold. The leveling off is due to the limited snapshot data for training the quadratic manifold as well as the limited accuracy of the full-model snapshots.

Figure~\ref{fig:BurgersErrors}b shows the averaged relative error of the QMNG reduced solutions for $m \in \{256, 512, 1024\}$ uniformly sampled collocation points. The results indicate that the more collocation points are used in QMNG the more reliable the approximations become. At the same time, it can be noted that the QMNG reduced solution achieves a comparable error decay as the reconstruction error already with $m = 256$ collocation points, which is almost $10\times$ lower than the number of grid points in the full model. This is an indication that separating the collocation points from the grid points can be interpreted as a form of hyper-reduction, which is naturally integrated in Neural Galerkin schemes when applied to quadratic manifolds. 

We close this section by clearly stating that the implementation of our QMNG reduced model does not achieve speedups compared to the full model in this example. One reason is that the full model is over a one-dimensional spatial domain and thus cheap to solve. Another reason is that automatic differentiation incurs costs that are substantial in our implementation.

\begin{figure}
    \centering
    \includegraphics[width=\linewidth]{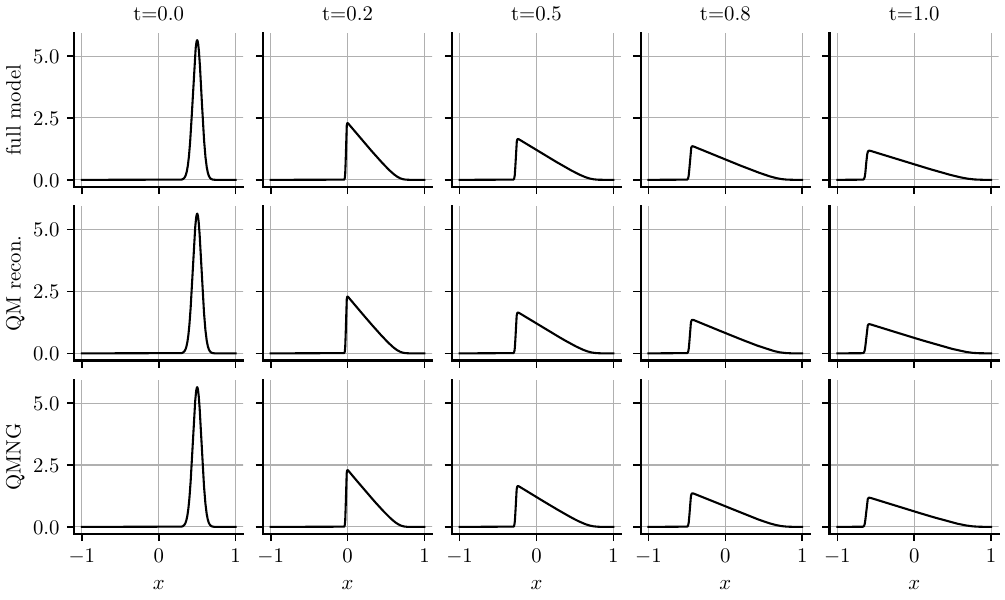}
    \caption{Burgers' equation: Velocity field snapshots from the full model of the Burger's problem, their reconstruction after projection onto the quadratic manifold and the corresponding snapshots from the interpolated QMNG reduced model of size $n=30$ with regularization parameter $10^{-6}$. The snapshots have been computed for the test parameter $\mu_3^{(\text{test})} = 0.5$.}
    \label{fig:BurgersSnapshots}
\end{figure}

\begin{figure}
    \centering
        \begin{subfigure}[T]{\textwidth}
        \centering
        \includegraphics{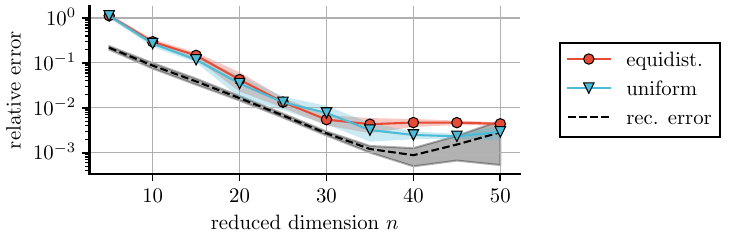}
        \caption{averaged relative error for $m = 1024$ equidistant and uniformly sampled collocation points}
    \end{subfigure}
    \hfill
    \begin{subfigure}[T]{\textwidth}
        \centering
        \includegraphics{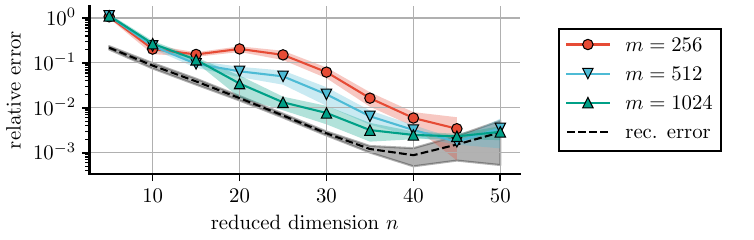}
        \caption{averaged relative error for $m \in \{256, 512, 1024\}$ uniformly sampled collocation points}
    \end{subfigure}
    \caption{Burgers' equation: Using different collocation points than the grid points in QMNG reduced models is a form of hyper-reduction. Using about $10\times$ fewer collocation points ($m = 256$) than the number of grid points of the underlying full model ($N = 2048$) still leads to an error decay that is similar to the decay of the reconstruction error on the quadratic manifold.}
    \label{fig:BurgersErrors}
\end{figure}

\section{Conclusions}\label{sec:Conc}
We applied Neural Galerkin schemes to quadratic parametrizations, which are more structured than neural-network parametrizations. In particular, the corresponding QMNG reduced models have locally unique solutions. In numerical experiments, the proposed QMNG reduced models achieve orders of magnitude speedups, including online efficiency for linear full models. A potential future research direction is to investigate how the time-varying test spaces given by Neural Galerkin schemes can be effectively utilized in non-intrusive settings, such as those considered in \cite{GeelenWW2023Operator}

\bibliographystyle{elsarticle-num}
\bibliography{bibliography, donotchange, main}

\end{document}